\documentclass[12pt]{amsart}
\usepackage{amssymb}
\usepackage{mathrsfs}
 
\theoremstyle{plain}

\newtheorem{theorem}{Theorem}[section]
\newtheorem{lemma}[theorem]{Lemma}
\newtheorem{proposition}[theorem]{Proposition}
\newtheorem{corollary}[theorem]{Corollary}

\newtheorem{example}[theorem]{Example}

\theoremstyle{definition}

\newcommand{\eps}{\varepsilon}

\newcommand{\dist}{{\rm dist\,}}

\newcommand{\spn}{{\rm span\,}}

\newcommand{\BB}{\mathcal B}

\newcommand{\EE}{\mathcal E}

\newcommand{\R}{\mathbb R}

\newcommand{\NN}{\mathbb N}

\newcommand{\SSS}{\mathcal S}

\newif\ifComplain
\Complaintrue           
\def\complain#1{\ifComplain\ifhmode \newline\fi{\sf *** \ \ #1
\\}\fi}

\newif\ifmarglab
\makeatletter
\def\label#1{\@bsphack\ifmarglab\marginpar{LAB:#1}\fi\if@filesw {\let\thepage\relax
   \def\protect{\noexpand\noexpand\noexpand}%
   \edef\@tempa{\write\@auxout{\string
      \newlabel{#1}{{\@currentlabel}{\thepage}}}}%
   \expandafter}\@tempa
   \if@nobreak \ifvmode\nobreak\fi\fi\fi\@esphack}
\makeatother
\marglabfalse

\usepackage{graphicx}

\long\def\onefigure#1#2{
\begin{figure*}[tbp]
\begin{center}
#1
\end{center}
\caption{#2}
\end{figure*}
} 

\newcommand\newipefig[2]
{\onefigure{\includegraphics{#1.pdf}}{\label{#1} #2}}

\begin{document}

\title[Characterization]{When products of projections diverge}

\author[E. Kopeck\'a]{ Eva Kopeck\'a}
\address{Department of Mathematics\\
   University of Innsbruck\\
 A-6020 Innsbruck, Austria}

\email {eva.kopecka@uibk.ac.at}

\thanks{Research partially supported by  NSF  Grant No. 1440140 while  the author was visiting  MSRI, Berkeley, California in 2017}



\begin{abstract}
Slow convergence of cyclic projections implies divergence of random projections
and vice versa.

Let $L_1,L_2,\dots,L_K$ be a family of $K$
closed   subspaces of a Hilbert space. It is well known that although the cyclic product of the orthogonal projections on these spaces always converges in norm, random products might diverge. Moreover, in the cyclic case there is a dichotomy: the convergence is fast if and only if 
$L_1^{\perp}+\dots+L_K^{\perp}$ is closed;
 otherwise the convergence is arbitrarily slow. 
 
 We prove a parallel to this result concerning random products: 
we characterize  those families $L_1,\dots,L_K$ for which   
all random products converge using their geometric and combinatorial structure.
\end{abstract}

\maketitle


\section*{Introduction}

Let $K$ be a fixed natural number and let $L_1,L_2,\dots,L_K$ be a family of $K$
closed   subspaces of a Hilbert space  $H$ such that, for simplicity, $\bigcap_{i=1}^K L_i=\{0\}$.
Let $z_0\in H$ and $k_1,k_2,\dots \in [K]=\{1,2,\dots,K\}$
be an arbitrary sequence in which each $k\in[K]$ appears infinitely often. Consider the sequence of vectors $\{z_n\}_{n=1}^{\infty}$ defined by
\begin{equation}\label{iter}
z_{n}=P(L_{k_n})z_{n-1},
\end{equation}
where $P(X)$ denotes the orthogonal projection of $H$ onto the subspace  $X$.
Then    $\{z_n\}$  is a  weakly-null sequence    according to \cite{AA}.  If $H$ is infinite dimensional and $K\geq 3$, then 
the sequence   $\{z_n\}$  does not, in general, converge in norm as exhibited in  \cite{KM, KP}.

If  the sequence
$\{k_n\}$ is periodic, the product of the projections is called cyclic. In the cyclic case  the sequence $\{z_n\}$  always does converge  in norm  \cite{N, Ha}.   
The convergence is, however,  either exponentially fast, or arbitrarily slow, depending on whether $L_1^{\perp}+\dots+L_K^{\perp}$ is closed or not. This dichotomy result was obtained in \cite{BDH, DH} and independently in   \cite{BaGM1, BaGM2}; see also \cite{BaS1, BaS2} for recent development. In this paper we show how slow convergence of cyclic iterates of projections corresponds to existence of non-converging iterates.

For the general, not necessarily cyclic case Bauschke showed in \cite{B}    that 
if $\sum_{j\in J}L_j^{\perp}$ is closed for each nonempty 
$J\subset [K]$, then $\{z_n\}$ converges in norm. In Example~\ref{slownono} we show that  a condition this strong is not necessary for the convergence.
 
Deutsch and Hundal ask in \cite{DH} if just
$L_1^{\perp}+\dots +L_K^{\perp}$
being closed 
is enough to guarantee the norm convergence of $\{z_n\}$ 
in the general case. In Theorem~\ref{3-eq-positive} we show that this is indeed the case for $K=3$ and $K=4$.
For $K\geq 5$, we exhibit in Example~\ref{5}
that this is {\em not} the case in general.
For $K\geq 4$, we give in Theorem~\ref{hypothesis} 
a condition which is equivalent to
 the norm convergence  of $\{z_n\}$ and which involves connected subgraphs of the Johnson graph $J(K,4)$. 

Here is an  overview of the various conditions and their relations; for simplicity we assume that $\bigcap L_k=\{0\}$.

\begin{enumerate}
\item[(i)] $L_1^{\perp}+L_2^{\perp}+\dots +L_K^{\perp}$ is   closed in $H$;
\item[(ii)]  if $T=P(L_K)\dots P(L_2)P(L_1)$, then $\|T^n\|\leq Cr^n$ for some $r\in [0,1)$;
\item[(iii)]   for all closed subspaces  $\tilde L_i\subset L_i\subset H$,  each starting point $z_0\in H$  and every sequence of indices $k_1, k_2,\dots \in \{1,\dots,K\}$, the sequence of iterates defined by $z_n= P(\tilde L_{k_n}) z_{n-1}$ does  converge in norm.
\end{enumerate} 
For $K=2$ condition (iii) is always satisfied \cite{N}.
If $K\geq 2$, then (i)$\Leftrightarrow$(ii)  by \cite{BDH, DH, BaGM1, BaGM2}.
In Corollary~\ref{inspaces} we  prove that 
(iii)$\Rightarrow$(i) holds for $K\geq 3$ subspaces.
In  Theorem~\ref{3-eq-positive} we show that 
(i)$\Leftrightarrow$(iii)   for $K=3$ and $K=4$ subspaces.

Here is the complementary version of the above, enriched by the Johnson graphs condition (d) and examples (e)   of subspaces that show that the condition (d) really  takes place.
\begin{enumerate}
\item[(a)] $L_1^{\perp}+L_2^{\perp}+\dots +L_K^{\perp}$ is   not closed in $H$;
\item[(b)] for every sequence $\{a_n\}\in c_0$ there exists 
a starting point $z\in H$ so that    $|T^nz-P(\bigcap L_k)z|\geq a_n$;
\item[(c)]   there exist  closed subspaces  $\tilde L_i\subset L_i\subset H$,   a  starting point $z_0\in H$  and  a  sequence of indices $k_1, k_2,\dots \in \{1,\dots,K\}$ so that   the sequence of iterates defined by $z_n= P(\tilde L_{k_n}) z_{n-1}$ does not converge in norm;
\item[(d)] there exists $V\subset \binom{[K]}{4}$ such that $\bigcup V=[K]$, for $E=\{\{A,B\}\in \binom V2:\, |A\cap B|=3\}$ the graph $G=(V,E)$ is connected, and if $A\in V$, then $\bigcap_{i\in A} L_{i}$ is infinite dimensional, or $\sum_{i\in A} L_{i}^{\perp}$ is not closed;
\item[(e)]  
 for every $\emptyset\neq V\subset  \binom{[K]}4$  so that $\bigcup V=[K]$, and for every   $\alpha:V\to \{0,1\}$  there  exist closed subspaces $L_1,\dots,L_K\subset H$ with the following properties.  If $A\in V$ and $\alpha(A)=0$, then $\bigcap_{i\in A}L_{i}$ is infinite dimensional. If $A\in V$ and $\alpha(A)=1$, then $\sum_{i\in A} L_{i}^{\perp}$ is not closed.
If $B\in \binom{[K]}4\setminus V$, then $\bigcap_{i\in B}L_{i}$ is finite dimensional  and   $\sum_{i\in B} L_{i}^{\perp}$ is   closed.
\end{enumerate}  
If $K\geq 2$, then   (a)$\Leftrightarrow$(b)  by \cite{BDH, DH, BaGM1, BaGM2}.
In Corollary~\ref{inspaces} we  prove that 
(a)$\Rightarrow$(c) holds for $K\geq 3$ spaces.
In  Theorem~\ref{3-eq-positive} we show that 
(a)$\Leftrightarrow$(c)   for $K=3$ and $K=4$ spaces if $\bigcap L_k=\{0\}$.
In Theorem~\ref{hypothesis} we show that  (c)$\Leftrightarrow$(d) for $K\geq 4$. If $H$ is infinite dimensional and $K\geq 4$ then according to Proposition~\ref{johnbio}, (e) is satisfied. That is, 
for each induced subgraph $G$  of the Johnson graph $J(K,4)$, there are $K$ closed subspaces of an infinite dimensional Hilbert space with the geometry of all their 4-tuples corresponding to $G$.

The paper is organized as follows. In Section~\ref{almortho} we recall the closedness of $L_1^{\perp}+L_2^{\perp}+\dots +L_K^{\perp}$ as a geometric tool suitable  for examining the convergence of products of projections.
In Section~\ref{SJ} we consider subspaces $L_1,\dots ,L_K$ for which there is a non-convergent sequence of products of projections.
We use Johnson graphs to describe the family of sets of indices $A\subset [K]$ for which $\sum_{i\in A} L_{i}^{\perp}$ is not closed or 
$\bigcap_{i\in A}L_{i}$ is infinite dimensional, and  answer  a question of Deutsch and Hundal \cite{DH}. Let $K\geq 3$ and let $L_1^{\perp}+ \dots +L_K^{\perp}$ be not closed.
In Section~\ref{K34} we construct $\tilde L_k\subset L_k$, $k\in [K]$,  and a product of projections on the spaces $\tilde L_k$ which does not converge in norm.  For $K=3$ and $K=4$, we show  in Theorem~\ref{3-eq} that the existence of such $\tilde L_k$'s is equivalent to $L_1^{\perp}+ \dots +L_K^{\perp}$  not being closed.
In Section~\ref{K5} we give for $K\geq 5$ a weaker condition involving Johnson graphs
under which a non-converging product of projections exists.
In Section~\ref{examples}
we give examples showing where the previously exhibited results cannot be strengthened or modified.
In Section~\ref{appen}  we collect  a  couple of elementary tools used throughout  the  text.

{\bf Notation.} In the entire paper  $H$ is a Hilbert space; we will explicitly mention when we need $H$ to be infinite dimensional. 
For a closed subspace $X$ of $H$ we denote by $P(X)$  the orthogonal projection onto $X$.  
 For $M,N\subset H$ we denote by $\bigvee M$ the closed linear span  of $M$ and by $M\vee N$   the closed linear span of $M\cup N$. Similarly, we use $\vee x$ and $x\vee y$ for $x,y\in H$. By $|x|$ we denote the norm of $x$.

Let  $\BB(H)$ be the space of bounded linear operators from $H$ to $H$.
For $m\in\NN$, let $\SSS_m$ be the free semigroup with generators $a_1,\dots,a_m$. 
If $\varphi=a_{i_r}\cdots a_{i_1}\in\SSS_m$ (for some $r\in\NN$ and $i_j\in\{1,\dots,m\}$)
 and $A_1,\dots,A_m\in \BB(H)$, then we write
$\varphi(A_1,\dots,A_m)=A_{i_r}\cdots A_{i_1}\in \BB(H)$. If $X_1,\dots,X_m$ are closed subspaces of $H$, then, by a slight abuse of notation,  $\varphi(X_1,\dots,X_m)=P(X_{i_r})\cdots P(X_{i_1})\in \BB(H)$
is the corresponding product of orthogonal projections.

Denote by $|\varphi|=r$ the ``length" of the word $\varphi$  and by $|\varphi_i|$ or   $|\varphi_{a_i}|$ the number of ``occurrences" of $a_i$ in the word $\varphi$. Then $\sum_{i=1}^m |\varphi_i|=|\varphi|=r$.
 
We denote $[K]=\{1,\dots,K\}$.  For a finite set $V$ and $k\in \NN$ we denote by $\binom{V}k$ the set of all $k$-element subsets of $V$. 

\section{Common orthogonal structure}\label{almortho}

Suppose $L_1,\dots,L_K$ are closed subspaces of a Hilbert space $H$.
The existence of an orthonormal sequence which is {\em almost} contained  in all of 
the spaces  $L_k$ is the reason behind both the slow convergence of the cyclic products $(P(L_K)\dots P(L_1))^n$ and the possible non-convergence of the random products of $P(L_k)$'s.

Here is a brief heuristic.
The larger the acute angle between two lines $L_1$ and $L_2$, the faster do the projections $(P(L_2)P(L_1))^nz$ converge to the intersection of the lines. This simple geometric observation 
was developed  in \cite{BDH, DH,  BaGM1, BaGM2} to the statement
that if ``the angles"  between the subspaces $L_1,\dots,L_K$ of a Hilbert space $H$ are positive, then 
$(P(L_K)\dots P(L_1))^n$ converges ``fast". An economical way of stating that the ``angles are positive"
is to say that $L_1^{\perp}+ \dots+L_K^{\perp}$ is closed.   
Bauschke and Borwein \cite {BB}  have given an equivalent geometric condition, $\neg$(ii) of the next lemma, that $L_1,\dots,L_K$ are boundedly regular. 
In infinite dimensions even spaces which intersect only at the origin,
can have ``zero angle" between them. Geometrically this means
that the spheres of these spaces {\em almost touch} (condition (ii)  of Lemma~\ref{BB}), or, more vividly, that  the subspaces $L_1,\dots,L_K$ {\em almost share} an orthonormal sequence  (condition (iv)  of Lemma~\ref{BB}). 

On the other hand,  the iterative construction of  \cite{P, KM, KP} 
uses an orthonormal sequence  $\{w_n\}$ to build three {\em almost touching}  
spaces $\tilde L_1, \tilde L_2, \tilde L_3$ for which certain random product of projections does not converge. 
If $\{w_n\}$ is very close to all of the spaces $L_1, L_2, L_3$, it is possible to build $\tilde L_1$ within $L_1$, and in general $\tilde L_k\subset L_k$.

\begin{lemma}\label{BB}
Let $L_1,\dots, L_K$, $K\in \NN$  be closed subspaces of a Hilbert space $H$  and let $L=\bigcap_{k\in [K]} L_k$.
  
The following statements are equivalent:
\begin{enumerate}
\item[(i)] $L_1^{\perp}+ \dots+L_K^{\perp}$ is {\em not} closed in $H$;
\item[(ii)] there exists a bounded
$A\subset H$  and    
$\eps>0$ such that for all $\delta>0$, there exists   
$w\in A$ 
so that $\max\{\dist(w,L_i):\, i\in  [K]\}< \delta$ and $\dist(w, L)>\eps$;
\item[(iii)] there exists a normalized weakly-null sequence $\{w_n\}$ in $L^{\perp}$   so that 
  $\lim_{n\to\infty}\max\{\dist(w_n,L_i):\, i\in  [K]\}=0$;
 \item[(iv)] there exists an orthonormal sequence $\{w_n\}$ in $L^{\perp}$   so \\ that 
  $\lim_{n\to\infty}\max\{\dist(w_n,L_i):\, i\in  [K]\}=0$. 
 If $V\subset H$ is finite dimensional,  then it is, moreover,  possible to choose  $\{w_n\}$ in $V^{\perp}$. 
\end{enumerate}
\end{lemma}
\begin{proof}
The equivalence of (i) and (ii) appears in \cite{BB}. If the negation of the property (ii) takes place, then the $K$-tuple of the subspaces  $L_1,\dots, L_K$ is in \cite{BB} called  boundedly regular.

Clearly, (iv) implies (ii) and (iii). Now assume (ii) is satisfied.  To show (iii), 
for $\delta_n=1/n$, choose  
$w_n\in A$ so that $\max\{\dist(w_n,L_i):\, i\in \{1,\dots ,K\}\}< 1/n$ and $\dist(w_n, L)>\eps$.   Since $w_n-P(L)w_n$ also satisfies these inequalities, we can assume that $w_n\in L^{\perp}$. Since 
$A$ is bounded and  
the $L_i$'s are linear subspaces, by normalizing we can assume
that $|w_n|=1$ for all $n$ and  that $\{w_n\}$ converges weakly to some $w\in L^{\perp}$. 
Suppose $w\neq 0$. Then there is a weakly open halfspace $W\ni w$ so that, without loss of generality, $\bar W\cap L_1=\emptyset$. This means that $\lim_{n\to\infty} \dist(w_n,L_1)>0$, which is a contradiction. 

Lemma~\ref{justforme0} implies that (iv) follows from (iii).
\end{proof}

We will also use the followig simple corollary.

\begin{lemma}\label{seqeq}
Let $H$ be a Hilbert space and $L_1,\dots, L_K$, $K\in \NN$  its closed subspaces and let $L=\bigcap_{k\in [K]} L_k$.
The following statements are equivalent:
\begin{enumerate}
\item[(i)] $L$ is infinite dimensional, or $L_1^{\perp}+ \dots+L_K^{\perp}$ is {\em not} closed in $H$;
\item[(ii)] there exists a normalized weakly-null sequence $\{w_n\}$ in $H$   so that 
  $\lim_{n\to\infty}\max\{\dist(w_n,L_i):\, i\in  [K]\}=0$;
 \item[(iii)] there exists an orthonormal sequence $\{w_n\}$ in $H$   so \\ that 
  $\lim_{n\to\infty}\max\{\dist(w_n,L_i):\, i\in  [K]\}=0$. 
 If $V\subset H$ is finite dimensional,  then it is, moreover,  possible to choose  $\{w_n\}$ in $V^{\perp}$. 
\end{enumerate}
\end{lemma}
\begin{proof}
Assume (i) is satisfied. If $L$ is infinite dimensional, then it contains an orthonormal sequence, hence (ii) and (iii) are satisfied.
If $L_1^{\perp}+ \dots+L_K^{\perp}$ is not closed, then (ii) and (iii) are satisfied according to Lemma~\ref{BB}. Clearly, (iii) implies (ii).

Assume that (ii) is satisfied and $L$ is finite dimensional. We write 
$w_n=x_n+y_n$, where $x_n\in L$ and $y_n\in L^{\perp}$. We may assume that $\{x_n\}$ converges in norm to $x\in L$, and $\{y_n\}$ converges weakly to $y\in L^{\perp}$. Then $\{w_n\}$ converges weakly to $0=x+y$; hence 
$x=y=0$ and $\lim_{n\to\infty}\max\{\dist(y_n/|y_n|,L_i):\, i\in  [K]\}=0$. By the implication (iii)$\Rightarrow$(i) of Lemma~\ref{BB}, $L_1^{\perp}+ \dots+L_K^{\perp}$ is not closed.
\end{proof}

Here is an explanation, where  normalized weakly-null sequences in $L^{\perp}$ naturally appear in connection to non-convergent products of projections.

Let $L_k$, $k\in [K]$, be as in the assumptions
of Lemma~\ref{seqeq}. We write each $L_k=L+X_k$, where $X_k$ is a closed subspace of $L^{\perp}$.
Then $P(L_k)=P(L)+P(X_k)$ and $\bigcap_{k\in [K]} X_k=\{0\}$.  Let $z_0=y+w_0$ for $y\in L$ and $w_0\in L^{\perp}$.
Assume that $\{z_n\}$ is defined as in (\ref{iter}).
Then $z_n=y+w_n$, where $w_n\in L^{\perp}$ is defined iteratively by 
$$
w_n=P(X_{k_n})w_{n-1}=P(L_{k_n})w_{n-1}.
$$
If $\{z_n\}$ does not converge in norm, then neither does $\{w_n\}$. By \cite{AA}, $\{w_n\}$  is weakly null. Since the sequence $\{|w_n|\}$ is decreasing,   there exists $\lim_{n\to \infty}|w_n|=c>0$.
By considering $z_0/c$ instead of $z_0$, we can assume that $\lim_{n\to \infty}|w_n|=1$.

\section{Divergent products and Johnson graphs}\label{SJ}

In this section we consider subspaces for which there is a divergent sequence of products of projections, and examine when  the sum of the orthogonal complements of the subspaces  is not closed.

Let $L_1,\dots, L_K$,   be closed subspaces of a Hilbert space $H$. Suppose there is a starting point $z_0\in H$  and a sequence of indices $k_1, k_2,\dots \in [K]$ containing each of the numbers $1,\dots, K$ infinitely  often so that the sequence of iterates defined by $z_n= P(L_{k_n}) z_{n-1}$ does not converge in norm.  By \cite{N}, $K\geq 3$.
If $K=3$ or $K=4$,  
Corollary~\ref{n34} shows that $L_1^{\perp}+\dots+L_K^{\perp}$ is not closed, answering thus positively a question of Deutsch and Hundal \cite{DH}. For $K\geq 5$, this is in general no longer true, but still
$\sum_{i\in A} L_{i}^{\perp}$ is not closed for a rich family of sets
$A\subset [K]$. To describe the structure of this  family, Johnson graphs turn out to be the right tool.

Let $K\geq 4$ be a natural number. 
The vertices of the Johnson graph $J(K,4)$ are the 4-element subsets of [K]; two 
vertices are adjacent when the intersection of the two vertices (subsets) contains 3 elements.

In Theorem~\ref{johnson} we show that there is a connected induced
subgraph $G=(V,E)$ of $J(K,4)$ so, that $\bigcup V=[K]$ and if $A\in V$, then $\sum_{i\in A} L_{i}^{\perp}$ is not closed or $\bigcap_{i\in A} L_{i}$ is infinite dimensional.
Conversely, in Proposition~\ref{johnbio} we construct for   each subgraph $G$ as above   a corresponding  configuration of subspaces of $H$.

  Next we prove   that a diverging sequence of iterates of projections needs at least three subspaces to ``move along".

Let $L_1,\dots, L_K$   be closed subspaces of a Hilbert space $H$, $z_0\in H$  and $z_n= P(L_{k_n}) z_{n-1}$ for a sequence of indices $k_1, k_2,\dots \in [K]$.  We denote 
\begin{equation}\notag
I(m,\delta)=\{k\in [K]:\, B(z_{m},\delta)\cap L_k\neq \emptyset\}.  
\end{equation}

\begin{lemma}\label{3step}
Let $L_1,\dots, L_K$  be closed subspaces of a Hilbert space $H$ so that $\bigcap L_k=\{0\}$. Suppose there is a starting point $z_0\in H$  and a sequence of indices $k_1, k_2,\dots \in [K]$ containing each of the numbers $1,\dots, K$ infinitely  often so that the sequence of iterates defined by $z_n= P(L_{k_n}) z_{n-1}$ does not converge in norm.  
Then for every $\delta>0$, there exists $M\in \NN$ so that 
$
|I(m,\delta)|\geq 3
$
for every $m\geq M$.
\end{lemma}
\begin{proof}
The orthogonal projection is linear, the sequence $\{|z_n|\}$ is decreasing, on the other hand the sequence 
$\{z_n\}$ does not converge in norm, so we can assume that 
$\lim |z_n|=1$. Since $\bigcap L_k=\{0\}$, according to  \cite{AA}, the sequence $z_n$ converges weakly to zero. 

Since $\{z_n\}$ does not converge in norm, 
$$
0<A:=\lim_{N\to\infty}\sup\{|z_n-z_m|:\, N\leq m\leq n\}.
$$
We can assume $0<\delta<A/2$ and choose $M$ so that 
$|z_i|^2<1+\delta^2/4$ for all $M\leq i$.
Then  
\begin{equation}\label{1step}
|z_i-z_{i+1}|^2=|z_i|^2-|z_{i+1}|^2\leq \delta^2/4.
\end{equation}
For a given $m>M$ choose the smallest $n>m$ so that 
$|z_n-z_m|>\delta/2$. Then $|z_n-z_m|<\delta$ by (\ref{1step}), as otherwise 
$$
|z_{n-1}-z_m|\geq |z_n-z_m|-|z_{n-1}-z_n|>\delta-\delta/2.
$$
Hence $z_m, z_{m+1},\dots,z_n\in B(z_m,\delta)$.
 If all of the points $z_m, z_{m+1},\dots,z_n$ consist of projections on only two of the subspaces $L_1,\dots, L_K$, then by {\em e.g.} \cite{KR}, we have 
\begin{equation}\notag
|z_n-z_m|^2\leq |z_n|^2-|z_{m}|^2\leq \delta^2/4,
\end{equation} 
 which is a contradiction.
\end{proof}

If $K\geq 4$ spaces are involved infinitely  often in a diverging sequence of iterates of projections, then  at least four of the spaces
have to be arbitrarily close together. 
In  Example~\ref{5} we show that for five  spaces this is no longer true.

\begin{proposition}\label{34}
Let $L_1,\dots, L_K$   be closed subspaces of a Hilbert space $H$, denote   $L=\bigcap L_k$. Suppose there is a starting point $z_0\in H$  and a sequence of indices $k_1, k_2,\dots \in [K]$ containing each of the numbers $1,\dots, K$ infinitely  often  so that the sequence of iterates defined by $z_n= P(L_{k_n}) z_{n-1}$ does not converge in norm.
Then there is $I\subset [K]$ with $|I|=\min\{4,K\}$ and a 
 normalized weakly-null sequence $w_n$ in $L^{\perp}$ so that 
$$
\lim_{n\to \infty}\max\{\dist(w_n,L_k):\, k\in I\}=0.
$$
\end{proposition}
\begin{proof}
As explained at the end of Section~\ref{almortho}, we can assume that 
$\{z_n\}$ is a weakly-null sequence in $L^{\perp}$, and  
$\lim |z_n|=1$.  
Notice that according to von-Neumann's theorem, only $K\geq 3$  
is possible. 
If $K=3$, then  by Lemma~\ref{3step} it is enough to put   $w_n=z_n/|z_n|$.

Suppose $4\leq K$. 
Let $0<\delta$ and $N\in \NN$ be given.
We will show that  there is 
  $m(\delta,N)>N$ so that 
  $4\leq |I(m(\delta,N),2\delta)|$. Indeed, we can assume that for a given $\delta$ the given $N$ is at least as large as the $M$ of Lemma~\ref{3step} and that for all $i\geq N$, we have 
$$
|z_i-z_{i+1}| \leq \delta.
$$ 
For a contradiction assume that $|I(m,2\delta)|=3$ for all $m\geq N$.
Since for all $m\geq N$,
$$
B(z_{m},\delta) \cup B(z_{m+1},\delta)\subset B(z_{m},2\delta),
$$
$3=|I(m,\delta)|=|I(m+1,\delta)|\leq|I(m,\delta)\cup I(m+1,\delta)|\leq |I(m,2\delta)|=3$. Then, without loss of generality, 
$$
\{1,2,3\}=I(N,\delta)=I(N+1,\delta)=I(N+2,\delta)=\dots,
$$
which means that $4$ appears in the given sequence $ \{k_i\}$ only finitely many times. This contradicts our assumptions.

Consequently, we can choose an increasing sequence $m_n\geq n$ so that $4\leq |I(m_n,1/n)|$.
Since there are only finitely many 4-tuples of $K$ numbers, hence by dropping to a subsequence we can assume that for all $n\in \NN$,
$$
\{1,2,3,4\}\subset  I(m_n,1/n)
$$
and define $w_n=z_{m_n}/|z_{m_n}|$. 
\end{proof}

Deutsch and Hundal ask in \cite{DH} if  $L_1^{\perp}+\dots+L_K^{\perp}$ being closed 
is enough to guarantee the norm convergence of  any product of projections onto the $L_k$'s.  
The next corollary    shows that  for $K=3$ and $K=4$ this is indeed the case.
In Theorem~\ref{3-eq}   we modify  the two conditions in the  corollary so that they are actually equivalent.

 \begin{corollary}\label{n34}
 Let $H$ be a Hilbert space, $K=3$ or $K=4$,   
and let $L_1,\dots,L_K$ be closed subspaces of $H$.
Suppose 
\begin{enumerate}
\item[(i)]   there exist   a starting point $z_0\in H$  and a sequence of indices $k_1, k_2,\dots \in  [K]$ so that the sequence of iterates defined by $z_n= P(L_{k_n})z_{n-1}$ does not converge in norm.
\end{enumerate} \newpage
Then
\begin{enumerate}
\item[(ii)] $L_1^{\perp}+\dots+L_K^{\perp}$ is {\em not} closed in $H$.
\end{enumerate}
 \end{corollary}
 \begin{proof}
 Denote $L=\bigcap L_i$ and assume (i). By Proposition~\ref{34}, there is a 
 weakly-null sequence $w_n$ in $L^{\perp}$ so that 
$$
\lim_{n\to \infty}\max\{\dist(w_n,L_k):\, k\in I\}=0
$$
and (ii) follows from 
   Lemma~\ref{BB}.
 \end{proof}

The following theorem strengthens Proposition~\ref{34}; it describes  what exactly is happening when $K\geq 4$.
Notice, that the graph $G$ in (ii) is the induced subgraph with vertices $V$ of the Johnson graph $J(K,4)$.  In Proposition~\ref{johnbio} we will construct for   each subgraph $G$ as above   a corresponding  configuration of closed subspaces of $H$. 
In Theorem~\ref{hypothesis} we will show that the  statement
(i) and (ii) are actually equivalent.   

\begin{theorem}\label{johnson}
Let $H$ be a Hilbert space, $4\leq K$,   
and let $L_1,L_2,\dots, L_K$ be closed subspaces of $H$.

Suppose 
\begin{enumerate}
\item[(i)]
 there exist a starting point $z_0\in H$, a sequence of indices \\
 $k_1, k_2,\dots \in [K]$ containing each of the numbers $1,\dots, K$ infinitely often, and    for each $k\in [K]$, there exists   a closed subspace  $\tilde L_k\subset L_k\subset H$,   so that the sequence of iterates defined by $z_n= P(\tilde L_{k_n}) z_{n-1}$ does not converge in norm.
\end{enumerate}
Then    
\begin{enumerate}
\item[(ii)]
there exists $V\subset \binom{[K]}{4}$ so that 
$\bigcup V=[K]$, for $E=\{\{A,B\}\in \binom V2:\, |A\cap B|=3\}$ the graph $G=(V,E)$ is connected, and if $A\in V$ then $\bigcap_{i\in A} L_{i}$ is infinite dimensional or   $\sum_{i\in A} L_{i}^{\perp}$ is not closed.
\end{enumerate}
\end{theorem}
\begin{proof}
Denote this time $L=\bigcap_{k\in [K]} \tilde L_{k}$.
As explained at the end of Section~\ref{almortho}, we may assume that 
$\{z_n\}$ is a weakly-null sequence in $L^{\perp}$, and  that 
$\lim |z_n|=1$.  
Define $N_{0}=1$.
For $\delta=1/i$ choose $M_i>N_{i-1}$ as in Lemma~\ref{3step} applied to $\tilde L_k$'s  and so  that $|z_m-z_{m+1}|<1/i$ for $M_i\leq m$. 
Let $N_{i}>M_i$ be the smallest number for which 
\begin{equation}\label{sm}
\{k_n\}_{n=M_i}^{N_i}=[K].
\end{equation}
Let 
\begin{equation}\notag
\begin{split}
V_i^m&=\{ A\subset I(m,2/i):\,  3\leq |A|\}, \\
V_i&=\bigcup_{m=M_i}^{N_i} V_i^m, \\
E_i&=\{\{A,B\}\in \binom {V_i}2:\,  A\neq B \mbox{ and } 3\leq |A\cap B|\}.
\end{split}
\end{equation}
Then by   (\ref{sm}),
\begin{equation}\label{union}
\bigcup V_i=[K].
\end{equation}
Consider the graph $G_i=(V_i,E_i)$. 
Clearly, each of its induced subgraphs $G^m_i=(V^m_i,E_i\cap \binom{V_i^m}2)$ is connected. 
   Since $3\leq |I(m, 1/i)|$ for all $m\geq M_i$ and  
   $I(m+1,1/i)\subset I(m,2/i)$, the 
graph $G_i=(V_i,E_i)$ is connected as well. Indeed, if $A\in V_i^m$ and $B\in V_i^{m+1}$ are given,  choose 
$D\in V_i^{m+1}$, $D\subset I(m+1, 1/i)$. Since all $G^m_i$'s are connected,  there is a path from $A$ 
to $D$ and from $B$ to $D$. 

Since $K$ is fixed, there are only finitely many such graphs $G_i$. By passing to a subsequence we can assume 
they are all the same $G=G_1=G_2=\dots$. Since $K\geq 4$, the graph $G$ has at least 2 vertices.
  Suppose   $\{A,B\}$ is an edge of $G$ so that  $|A|\leq |B|$. If $|A|=3$, then 
$A\subset B$ and $|B|\geq 4$.   Hence we can drop all vertices of size 3 from the graph and all edges containing them and the graph still fulfills  (\ref{union}) and stays connected.
Since with each set $A$ of size at least 5, $V$ contains all of its 4-element subsets, by keeping in $V$ only the sets of size 4  we preserve 
(\ref{union}). We claim that such a graph is still connected.
Suppose $P$ is a path from $C$ to $D$ in $G$ and $A$ is a vertex of $P$ with $|A|\geq 5$. Then $A$ can be replaced by at most three of its $4$-point subsets to create a walk from $C$ to $D$.
Hence we can drop from $V$ all  the vertices of size at least 5 and still preserve both connectivity and (\ref{union}), as claimed.

If $A\in V$, then 
$$
\lim_{n\to\infty}\max\{\dist(z_{m_n}, L_i):\, i\in A\}\leq \lim_{n\to\infty}\max\{\dist(z_{m_n},\tilde L_i):\, i\in A\}=0
$$
for a subsequence $\{z_{m_n}\}$ of the weakly-null sequence $\{z_{m}\}$. Hence  by
 Lemma~\ref{seqeq}, 
$\bigcap_{i\in A} L_{i}$ is infinite dimensional  or
$\sum_{i\in A} L_{i}^{\perp}$ is not closed.  
\end{proof}

We close  with ideas that will enable us in Section~\ref{K5} to prove that (i) and (ii)  in the above theorem are, in fact, equivalent.

\begin{lemma}\label{walk}
Let $G=(V,E)$ be the Johnson graph $J(K,4)$  and let 
$A_1,A_2,\dots\in V$, $\EE_i=\{A_i,A_{i+1}\}\in E$ be  a walk in $G$.
For a given $k\in [K]$ define $a_i=1$ if $k\in  A_i\cap A_{i+1}$ and 
$a_i=0$ otherwise.
Then the sequence $\{a_i\}$ is either constant, or it consists of blocks of ones separated by blocs of zeros   of    length at least two. 
\end{lemma}
\begin{proof}
We can assume that $k=1$ and that $a_i=1$ and $a_{i+1}=0$. 
Then $1\in A_{i+1}$.
Hence $1\notin A_{i+2}$, and $a_{i+2}=0$.
\end{proof}

Let $K\in \NN$ be given. Let $G=(V,E)$ be the Johnson graph $J(K,4)$, and let 
$A_1,A_2,\dots\in V$, $\EE_i=\{A_i,A_{i+1}\}\in E$ be  an infinite walk in $G$. With a slight abuse of notation we also denote 
by $\EE_i$ the intersection $A_i\cap A_{i+1}$.
We will   iteratively define $K$   sequences $\{s^k_i\}_{i\in \NN}$, $k\in [K]$,   consisting of the symbols $X,Y,Z,0$, so that the $i$-th entry of exactly one of the $K$ sequences is $X$, of exactly one is $Y$, of exactly one is $Z$, and the $i$-th entries of all the other sequences are $0$. We proceed according to the following rules.

To start with, we set $s^k_i=0$ for all $k\in [K]$ and $i\in \NN_{0}$. 
Assume $\EE_1=\{a,b,c\}$. 
In the first step we re-define   $s^a_1=X$,   $s^b_1=Y$,   $s^c_1=Z$.
Here the order does not matter; we simply choose one,  assuming,  for example, that  $a<b<c$.

In the $i$-th step we change  the $i$-th coordinate from zero to  $X$, $Y$, or $Z$, respectively, in   three of the sequences corresponding to $\EE_i$. We proceed as follows. Since $3\leq |\EE_{i-1}\cup\EE_{i}|\leq |A_i|=4$, we can
assume $\EE_{i-1}=\{a,b,c\}$  and $\EE_{i}=\{a,b,d\}$, with possibly $c=d$. We define $s^a_i=s^a_{i-1}$,   $s^b_i=s^b_{i-1}$,   $s^d_i=s^c_{i-1}$.
 Each of the $K$  resulting sequences starts with $0$ followed by, 
according to Lemma~\ref{walk},   blocks of $X$'s, blocks of $Y$'s and blocks of $Z$'s (some or all of which can be missing) separated by blocks of zeros of length at least two.
 For example: $0,X,X,X,0,0,Y,0,0,Y,Y,0,0,0,Z,Z,Z,\dots$.  
 
 Moreover, for each $i\in \NN$, there is exactly one $k_{i,X}\in [K]$ so  that $s_i^{k_{i,X}}=X$. Similarly, there is exactly one $k_{i,Y}\in [K]$ and exactly one $k_{i,Z}\in [K]$ so  that $s_i^{k_{i,Y}}=Y$ and $s_i^{k_{i,Z}}=Z$. Clearly, $\eps_i=\{k_{i,X},k_{i,Y},k_{i,Z}\}$.

\section{$K=3$ or $K=4$ spaces $L_1,\dots, L_K$}\label{K34}

Suppose  $K\geq 3$ and let $L_1^{\perp}+ \dots +L_K^{\perp}$ be not closed. The key technical result of this section is Proposition~\ref{squeorthoblocs}. With its help we will
  construct $\tilde L_k\subset L_k$, $k\in [K]$, and a product of projections on the spaces $\tilde L_k$ which does not converge in norm.  For $K=3$  and $K=4$, we will show  in Theorem~\ref{3-eq} that the existence of such $\tilde L_k$'s is equivalent to $L_1^{\perp}+ \dots +L_K^{\perp}$  not being closed.

The following result on the continuous dependence of words on letters  will enable us to build
the spaces $X$ and $Y$ of Lemma~\ref{orthobuildingbloc} using any normalized weakly-null sequence instead of an orthonormal basis.

\begin{lemma}\label{wordcont}\cite{KP} 
Let $\psi\in \SSS_n$ for some $n\in \NN$. Assume  $A_i,B_i,E\in  \BB(H)$, $i\in \{1,2,\dots,n\}$, are  contractions so that each $A_i$ commutes with $E$.   Then
$$
\|\psi(A_1,\dots, A_n)E -\psi(B_1,\dots, B_n)E\| \leq \sum_{1\leq i\leq n} |\psi_i| \ \|A_i E - B_i E\|.
$$
\end{lemma}

The following lemma expresses the essence of the construction of \cite{P}.

\begin{lemma}\label{orthobuildingbloc}
For every  $\eps>0$, there exists $N=N(\eps)  \in \NN$, 
so that for every $\eta>0$, there exists   $\psi\in \SSS_3$ such that $|\psi_1|\leq N\leq |\psi|$  
with the following property.

Given  subspaces 
 $X\subset E$ of a Hilbert space $H$
so that $X$ is separable, $\dim X\geq N$, and $\dim(X^{\perp}\cap E)\geq \dim X$  and given $u,v \in X$  so that $\|u\|=\|v\|=1$ and $u\perp v$,
there exists a subspace $Y\subset E$ such that
\begin{enumerate}
\item[(i)]  $X\cap Y=\{0\}$;
\item[(ii)]  $\|P(X)-P(Y)\|<\eta$;
\item[(iii)] $|\psi(W, X,Y)u-v|<\eps$, where $W=u\vee v$.
\end{enumerate}
If, moreover,    $\tilde X,\tilde Y, \tilde Z$ are subspaces of $H$ such that
\begin{enumerate}
\item[(iv)]   $\|P(W)-P(\tilde Z)P(E)\|\leq\eps/N$;
\item[(v)] $\|P(X)- P(\tilde X)P(E)\|\leq\eps /|\psi|$;
\item[(vi)]  $\|P(Y)- P(\tilde Y)P(E)\|\leq\eps /|\psi|$,
\end{enumerate}
  then
$$
\bigl\|\psi(\tilde Z, \tilde X ,\tilde Y )u-v\bigr\|\leq4\eps.
$$
\end{lemma}
\begin{proof}
The statements (i)-(iii) are  basically Lemma~2.3 of \cite{KP}. To  see that one can work with finite dimensional spaces and the dependence of the dimension $N$ on $\eps$, one just needs to read carefully the proofs of 
Lemma~2.1, Lemma~2.2, and Lemma~2.3 of \cite{KP}, or, alternatively,  Lemma~2.1, Lemma~2.4, and Lemma~2.5 of \cite{KM}.
To show we have a slight freedom of  choice concerning  the 3 spaces onto which we project, we use 
 Lemma~\ref{wordcont}:
\begin{equation}\notag
\begin{split}
\bigl\|&\psi(\tilde Z, \tilde X , \tilde Y )P(E)-\psi(W, X,Y)P(E)\bigr\|  \leq N\cdot \|P(W)-P(\tilde Z)P(E)\| \\
&+|\psi|\cdot\|P(X)-P(\tilde  X)P(E)\|+|\psi|\cdot\|P(Y)-P(\tilde Y)P(E)\| \\
& \leq 3\eps;
\end{split}
\end{equation}
hence
$$
|\psi(\tilde Z, \tilde X , \tilde Y )u-v|  \leq   |\psi(\tilde Z, \tilde X , \tilde Y )u- \psi(W, X,Y)u|   +|\psi(W, X,Y)u-v|
 \leq 4\eps.
$$
\end{proof}

A careful reading of the proof of Lemma~2.5 of \cite{KP} or of Theorem~2.6 of \cite{KM} implies the following result.

\begin{lemma}\label{orthoblocs}
For every sequence of positive numbers   $\eps_i>0$, $i\in \NN$, there exist  $N_i\in \NN$  and $\Psi^i\in \SSS_3$
  with the following property.

Let $\{e_i\}_{i=1}^{\infty}$ be an orthonormal sequence in a Hilbert space $H$ and $W_i=e_i\vee e_{i+1}$.  Let   $\{E_i\}_{i=1}^n$ be a sequence of   subspaces of $H$ so that 
\begin{enumerate}
\item[(i)]  $\dim E_i\geq 2N_i$;
\item[(ii)] $e_i, e_{i+1}\in E_i$;
\item[(iii)] $\vee e_{i+1}=P(E_i)E_{i+1}=P(E_{i+1})E_i$;
\item[(iv)] $E_i\perp E_j$ if $|i-j|\geq 2$.
\end{enumerate}
Then there exist  sequences  $X_i,Y_i$   of   subspaces of $H$ so that 
\begin{enumerate}
\item[(v)]   $X_i,Y_i\subset E_i$;
\item[(vi)]  $e_i, e_{i+1}\in X_i$;
\item[(vii)]   $|\Psi^i(Z, X,Y)e_i-e_{i+1}|<\eps_i$;
\end{enumerate}
where $Y_0=\vee e_1$  and
$
X=\bigvee_{i=1}^{\infty}  X_i,\  Y=\bigvee_{k=0}^{\infty} Y_{2k+1},\ Z=\bigvee_{k=0}^{\infty} Y_{2k}
$.
\end{lemma} 

We allow a slight skewing  in the above  rigid orthonormal structure.
This will enable us later to fit the spaces $X,Y,Z$ inside  given subspaces of $H$.

\begin{proposition}\label{squeorthoblocs}
For every sequence of positive numbers   $\eps_i>0$, $i\in \NN$, there exist  $N_i\in \NN$ and  $\psi^i\in \SSS_3$ so  that 
$N_1\leq N_2\leq \dots$ and $|\psi^1|\leq |\psi^2|\leq \dots$ with the following property.

Let $\{w_n\}$ be an orthonormal sequence in a Hilbert space $H$ and  $e_1=w_1,e_2,e_3,\dots$ be a subsequence of $\{w_n\}$.
 Let   $\{E_i\}_{i=1}^n$ be a sequence  of spans of consecutive, with one vector $e_i$ overlapping, blocks of $\{w_n\}$  so that 
\begin{enumerate}
\item[(i)]  $\dim E_i\geq 2N_i$;
\item[(ii)] $e_i, e_{i+1}\in E_i$;
\item[(iii)] $\vee e_{i+1}=P(E_i)E_{i+1}=P(E_{i+1})E_i$;
\item[(iv)] $E_i\perp E_j$ if $|i-j|\geq 2$.
\end{enumerate}
Then there exist  subspaces $X_i,Y_i$ of $H$ so that 
\begin{enumerate}
\item[(v)]   $X_i,Y_i\subset E_i$;
\item[(vi)]  $e_i, e_{i+1}\in \chi_i\subset \{w_n\}$, where $\chi_i$  is an orthonormal basis  of $X_i$;
\item[(vii)]   $|\psi^i(W_i, X_i,Y_i)e_i-e_{i+1}|<\eps_i$, where $W_i=e_i\vee e_{i+1}$.
\end{enumerate}
Suppose, moreover, that
 $S, T,U:\spn\{w_n\}\to  H$     are  linear   so that
\begin{enumerate}

\item[(viii)]  $\|P(R(E))-P(E)\|\leq (\eps_{i+1}/|\psi^{i+1}|)^2$   for all subspaces $E\subset E_i$ and all linear mappings $R$ so  that $R(w_n)\in \{S(w_n), T(w_n), U(w_n)\}$.

\item[(ix)] Let $\tilde X_i\subset S(E_i)$, $\tilde Y_i\subset T(E_i)$, $\tilde Z_i\subset U(E_i)$  and
let $\tilde X=\bigvee_{i=1}^{\infty} \tilde X_i$, $\tilde Y=\bigvee_{i=1}^{\infty} \tilde Y_i$, and $\tilde Z=\bigvee_{i=1}^{\infty} \tilde Z_i$.
\end{enumerate}
For $i\leq 0$ define $e_i=0$ and $X_i=Y_i=\{0\}$. \\
Assume that for a fixed $j\in \NN$,  
\begin{enumerate}
 \item[(x)]  $\tilde X_j=S(X_j)$;
 \item[(xi)] $\tilde Y_j=T(Y_j)$, $\tilde Y_{j-1}\subset \vee T(e_{j-1})$, and 
$\tilde Y_{j+1}\subset \vee T(e_{j+2})$;
\item[(xii)] There are four options for $\tilde Z_j$:
\begin{enumerate}
 \item[1)] $\tilde Z_j= \vee\{U(e_j),  U(e_{j+1})\}$, $\tilde Z_{j-1}=\tilde Z_{j+1}=\{0\}$;
 \item[2)] $\tilde Z_j= \vee U(e_j)$, $\tilde Z_{j-1}=\{0\}$, $\tilde Z_{j+1}=U(Y_{j+1})$ and $\tilde Z_{j+2}\subset \vee U(e_{j+3})$;
 \item[3)] $\tilde Z_j= \vee U(e_{j+1})$, $\tilde Z_{j+1}=\{0\}$, $\tilde Z_{j-1}=U(Y_{j-1})$ and $\tilde Z_{j-2}\subset \vee U(e_{j-2})$;
 \item[4)] $\tilde Z_j=\{0\}$,   $\tilde Z_{j-1}=U(Y_{j-1})$, $\tilde Z_{j-2}\subset \vee U(e_{j-2})$, and, similarly $\tilde Z_{j+1}=U(Y_{j+1})$ and $\tilde Z_{j+2}\subset \vee U(e_{j+3})$.
\end{enumerate}
\item[(xiii)] If $t\in\{j-1,j,j+1\}$, then $|\langle f,g\rangle|\leq (\eps_{j+1}/|\psi^{j+1}|)^2$ for all $|f|=|g|=1$ so that  $f\in E_t\cup S(E_t)\cup T(E_t)\cup U(E_t)$    and 
\begin{equation}\notag
g\in G_{R,t}=\bigvee_{w_n\notin E_t} R(w_n),
\end{equation} 
\end{enumerate}
where $ R(w_n)\in \{w_n, S(w_n), T(w_n), U(w_n)\}$. 
Then
\begin{equation}\label{oest}
|\psi^j(\tilde Z, \tilde X,\tilde Y)e_j-e_{j+1}|<30\eps_j.
\end{equation}
\end{proposition} 
\begin{proof} Assume $1>\eps_1>\eps_2>\dots$. 
We define $N_0=1$  and 
for each $i\in \NN$ we  choose $N_i=N(\eps_i)\geq N_{i-1}$  according to Lemma~\ref{orthobuildingbloc}. Then we  put $\eta_i=\eps_{i+1}/N_{i+1}$ for each $i\in \NN$.
Again, we use  Lemma~\ref{orthobuildingbloc} to choose $\psi^i\in \SSS_3$.
We define $X_i$ as an $N_i$-dimensional subspace of $E_i$ so that $e_i, e_{i+1}\in \chi_i\subset \{w_n\}$ for an orthonormal basis $\chi_i$ of $X_i$.
 We choose $Y_i\subset E_i$ according to Lemma~\ref{orthobuildingbloc}.

Since   
$
\tilde X=\tilde X_j\vee (\tilde X \cap G_{S,j})$ by (vi),  
Lemma~\ref{justforme1} and (xiii) imply that 
$$
\|P(\tilde X_j)+P(\tilde X\cap G_{S,j})-P(\tilde X)\|\leq 4  \eps_j/|\psi^j|.
$$
Hence by (viii) and (xiii),
\begin{equation}\label{X}
\begin{split}
\|P(\tilde X)P(E_j)-P(X_j)\| \leq &\|P(\tilde X_j)+P(\tilde X\cap G_{S,j})-P(\tilde X)\| \\ &+\|P(\tilde X_j)-P(X_j)\|+\|P(\tilde X\cap \tilde G_{S,j})P(E_j)\|\\
&\leq 6\eps_j/|\psi^j|.
\end{split}
\end{equation}
Similarly, by (xi), $\tilde Y=\tilde Y_j\vee (\tilde Y \cap G_{T,j})$; hence 
$$
\|P(\tilde Y_j)+P(\tilde Y\cap G_{T,j})-P(\tilde Y)\|\leq 4  \eps_j/|\psi^j|,
$$
and
\begin{equation}\label{Y}
\|P(\tilde Y)P(E_{j})-P(Y_{j})\|\leq 6\eps_{j}/|\psi^{j}|.
\end{equation}
It remains to prove that 
\begin{equation}\label{W}
\begin{split}
\|&P(\tilde Z)P( E_{j})-P(W_{j})\| 
\leq  13\eps_j/N_j,
\end{split}
\end{equation}
as then  (\ref{X}), and (\ref{Y}) and 
   Lemma~\ref{orthobuildingbloc} imply that
$$
\bigl\|\psi_j(\tilde Z, \tilde X,\tilde Y)e_j-e_{j+1}\|<30\eps_j.
$$
If 1) of (xii) is satisfied, then 
$\tilde Z=\tilde Z_j\vee (\tilde Z \cap G_{U,j})$ hence by (xiii) and Lemma~\ref{justforme1},
$$
\|P(\tilde Z_j)+P(\tilde Z\cap G_{U,j})-P(\tilde Z)\|\leq 4  \eps_j/|\psi^j|.
$$
Hence
\begin{equation}\notag
\begin{split}
&\|P(\tilde Z)P( E_{j})-P(W_{j})\| \\
&\leq \|P(\tilde Z_j)-P(W_j)\|\|P(E_j)\|+\|P(\tilde Z\cap G_{U,j})P(E_j)\| +4  \eps_j/|\psi^j|\leq 6  \eps_j/|\psi^j|
\end{split}
\end{equation}
by (viii) and (xiii).
 
In the other extreme case, when 4) of (xii) is satisfied, we have 
$\tilde Z=\tilde Z_{j-1} \vee \tilde Z_{j+1}\vee  (\tilde Z \cap G)$,
where 
$$
G=\bigvee_{w_n\notin E_{j-1}\cup E_j\cup E_{j+1} } U(w_n).
$$
Hence by (xiii) applied to $G_{U,t}$ for $t\in \{j-1,j,j+1\}$ and Lemma~\ref{justforme1}, 
$$
\|P(\tilde Z)-P(\tilde Z_{j-1})-P(\tilde Z_{j+1})-P(\tilde Z \cap G)\|\leq8\eps_{j}/|\psi^{j}|
$$
and 
$$
\|P(\tilde Z \cap G)P(E_j)\|\leq \eps_{j}/|\psi|^{j}.
$$

Since $X_{j-1}$ and $Y_{j-i}$ were chosen so that 
$\|P(X_{j-1})-P(Y_{j-1})\| 
  \leq  \eta_{j-1}$, by (viii),
\begin{equation}\notag
\begin{split}
\|&(P(X_{j-1})-P(\tilde Z_{j-1}))P(E_j)\| \\ &\leq
 \|P(X_{j-1})-P(Y_{j-1})\| 
 +\|(P(\tilde Z_{j-1})-P(Y_{j-1}))\| \\
 &\leq  \eta_{j-1}+(\eps_j/|\psi^j|)^2.
 \end{split}
\end{equation}
Similarly,
\begin{equation}\notag
\|(P(X_{j+1})-P(\tilde Z_{j+1}))P(E_j)\|\leq \eta_{j+1}+(\eps_j/|\psi^j|)^2.
\end{equation}
By the above four inequalities, 
\begin{equation}\notag
\begin{split}
\|&P(\tilde Z)P( E_{j})-P(W_{j})\|= \|(P(\tilde Z) - P(X_{j-1}+ X_{j+1}))P(E_{j})\|\\
& \leq 9\eps_{j}/|\psi^{j}|+\|(P(\tilde Z_{j-1})+P(\tilde Z_{j+1})  - P(X_{j-1})-P( X_{j+1}))P( E_{j}) \| \\ 
\leq & \eta_{j-1}+\eta_{j+1} +11\eps_{j}/|\psi^{j}|
\leq  2\eps_j/N_{j}+11\eps_{j}/|\psi^{j}|
\leq  13\eps_j/N_j.
\end{split}
\end{equation}
In 2) and 3) of (xii) one proceeds similarly.
 \end{proof}

The following corollary is interesting in particular in the case 
where $L=\bigcap_{k\in [K]} L_k$ is finite dimensional.
If it is infinite dimensional, the existence of the spaces $\tilde L_k\subset L$ was established  in \cite{KM}.

\begin{corollary}\label{inspaces}
Let $H$ be a Hilbert space and $\{w_n\}_{n\in \NN}$ a normalized weakly-null sequence in $H$. Let $K\geq 3$ and let  $L_1,\dots, L_K$  be closed subspaces of $H$ so that 
$$
\lim_{n\to \infty}\max\{\dist(w_n,L_k):\, k\in \{1,\dots,K\}\}=0.
$$
Then there exist closed subspaces $\tilde L_k\subset L_k $, a starting point $z_0\in H$,  and a sequence of indices $k_1, k_2,\dots \in  [K]$ containing each $i\in {K}$ infinitely  often  so that the sequence of iterates defined by $z_n= \tilde L_{k_n} z_{n-1}$ does not converge in norm.
\end{corollary}
\begin{proof} 
According to Lemma~\ref{seqeq}, 
we may assume that $\{w_n\}$ is an orthonormal sequence.
For   $\eps_i=200^{-i}$  we choose $N_i\in \NN$, $\psi^i\in \SSS_3$   according to  Proposition~\ref{squeorthoblocs} and 
  define
$\beta_n=(\eps_{n+1}/|\psi^{n+1}|)^2/8$.
We choose a decreasing sequence $0<\alpha_n< 1/(K\cdot200^n)$ of numbers small enough to   accommodate the   property described below and
so that $2\sum_n^{\infty} \alpha_i<\beta_n$.
 
Assume $x_n\in H$ satisfy $|w_n-x_n|\leq \alpha_n$ for all $n\in \NN$ and 
  define a linear mapping $S$ by
 $S(w_n)=x_n$. Then  
\begin{equation}\label{8}
\|P(S(E))-P(E)\|\leq  \beta_n
\end{equation}
 for all subspaces $E\subset H_n:=\bigvee_{j=n}^{n+2N_n}w_j$ and for all $n\in \NN$. 
 
 By taking a subsequence, we call it $\{w_n\} $ again,  we can assume that for a suitably choosen $w^k_n\in L_k$
\begin{equation}\label{app}
|w_n-w^k_n|\leq  \alpha_n
\mbox{  for all } n\in \NN \mbox{ and } k\in [K].
\end{equation}  
Then, 
 clearly, (\ref{8}) is also satisfied  if $n\leq M$ and  $E\subset \bigvee_{j=M}^{M+2N_n}w_j$. 

Assume, moreover, we have chosen the subsequence  so that 
$|\langle w_n,g\rangle|\leq  \alpha_n$ for all
$$
g\in G_n=\bigvee \{w_i, w_i^k:\, i<n,\, k\in [K]\}
$$
with $|g|\leq 1$ and all $n\in \NN$. This is possible, because $G_n$ is finite dimensional and $\{w_n\}$ converges weakly to zero.
Suppose $x_j\in \{w_j, w_j^k:\,  k\in [K]\}$, $j\in \NN$. Then according to Lemma~\ref{justforme3} and the choice of $\{\alpha_n\}$ we have that $|\langle f,g\rangle|\leq  2\beta_n$ for all $f\in  \bigvee_{n \leq j} x_j$ and $g\in G_n$ with $|f|=|g|=1$.
 If $n<N$, then by Lemma~\ref{justforme1},
\begin{equation}\label{13}
|\langle f,g\rangle| \leq  8\beta_n \mbox{ if } f\in  \bigvee_{n \leq j \leq N}x_j,\   g\in G_n\vee   \bigvee_{N<l}x_l,\mbox{ and } |f|=|g|=1. 
\end{equation}
In order to apply  Proposition~\ref{squeorthoblocs}, we define $e_1=w_1$ and
we  group  $\{w_{n}\}$      into successive blocks of length $2N_1,2N_2,\dots$.   The neighboring blocks  overlap by exactly one vector which we call $e_2,e_3,\dots$.   The spans of the blocks we call $E_1,E_2,\dots$. By (\ref{app}) we have 
    $|P(L_k)e_n-e_n|\leq  \alpha_n$ for all $k\in  [K]$ and $n\in \NN$.

First assume $K=3$. 
We define linear mappings $S,T,U: \spn\{w_n\}\to H$
  so that 
\begin{equation}\notag
\begin{split}
S(w_n)=&w^1_{n} \\
T(w_n)=&w^2_{n}   \\
U(w_n)=&w^3_{n}.
\end{split}
\end{equation}
Then by (\ref{8}) and (\ref{13}),  assumption    (viii)  and the assumption (xiii)  of Proposition~\ref{squeorthoblocs}  are satisfied for any $j\in \NN$.  Choose $X_i,Y_i\subset E_i$ according to Proposition~\ref{squeorthoblocs} and for $i\in \NN$, define  $\tilde X_i=S(X_i)$
and 
$$
\tilde Y_i=
\begin{cases}
T(Y_i) \mbox{ if $i$ is odd,}  \\
\{0\} \ \ \ \mbox{ if $i$ is even,} 
\end{cases}
\tilde Z_i=
\begin{cases}
U(Y_i) \mbox{ if $i$ is even,}  \\
\{0\} \ \ \ \mbox{ if $i$ is odd,} 
\end{cases}
$$
except for $Z_1$ which we define as $\vee U(e_1)$.
Assumptions (x), (xi), and (xii) of Proposition~\ref{squeorthoblocs}
are satisfied for all odd $j\in \NN$. If we switch between $\tilde Y_j$  
and $\tilde Z_j$, they are also satisfied  for all even $j\in \NN$.
 
   Choose  closed subspaces
 $\tilde X,\tilde Y,\tilde Z\subset H$ according to  (ix) of Proposition~\ref{squeorthoblocs} and call them $\tilde L_1, \tilde L_2, \tilde L_3$. Then $\tilde L_k\subset L_k $. Write for short $R_j=\psi^j(\tilde Z,\tilde X,\tilde Y)$ if $j$ is  odd and $R_j=\psi^j(\tilde Y,\tilde X,\tilde Z)$ if $j$ is even.
 Since $\|R_j\|\leq 1$, we have by induction
 \begin{equation}\label{11}
 \begin{split}
 \|R_n&R_{n-1}\dots R_1 e_1 -e_{n+1}\| \\
 &\leq \|R_nR_{n-1}\dots R_2(R_1 e_1 -e_2)\|+ \|R_nR_{n-1}\dots R_2e_2-e_{n+1}\| \\
 &\leq 30\eps_1+\|R_nR_{n-1}\dots R_3(R_2 e_2 -e_3)\|+ \|R_nR_{n-1}\dots R_3e_3-e_{n+1}\| \\
& \leq 30\eps_1+30\eps_2+\dots+\|R_n e_n-e_{n+1}\|\leq 30(\eps_1+\dots+\eps_n)\\
& \leq 60(\eps_1+\dots+\eps_n)\leq 1/2
 \end{split}
 \end{equation}
 for all $n\in \NN$. Since $\{e_i\}$ is an orthonormal sequence, the norm-limit
 $\lim_{n\to \infty} R_nR_{n-1}\dots R_1 e_1$ does not exist.

If $K>3$ define $\tilde L_k=L_k$ for $k\in \{4,\dots, K\}$ and
if   $j$ is odd $R_j=\psi_j (\tilde Z,\tilde X,\tilde Y)P(L_K)\dots P(L_4)$; for even $j$ we define $R_j$  similarly. Then
\begin{equation}\notag
\begin{split}
|R_i e_i-e_{i+1}|\leq &|\psi^i(\tilde Z,\tilde X,\tilde Y)e_i-e_{i+1}| \\&+\|\psi^i(\tilde Z,\tilde X,\tilde Y)\|\cdot |P(L_K)\dots P(L_4)e_i-e_i|\leq 60\eps_i.
\end{split}
\end{equation}
The statement then again follows  from (\ref{11}).
\end{proof}

 We show next that  for $K=3$ and $K=4$  the existence of a diverging product  of projections  is equivalent to $L_1^{\perp}+ \dots +L_K^{\perp}$  not being closed.

Notice that for (i) of the next theorem to be true, $H$ has to be infinite dimensional.
If $L=\bigcap_{k\in [K]} L_k$ is finite dimensional, in particular   if $L=\{0\}$, then it is enough to require $\tilde L_k\subset L_k$ in (iii).
If $L$ is infinite dimensional this condition cannot be dropped, as there always exist 
 $\tilde L_k  \subset L$ as required  in (iii) \cite{KM, KP}.

\begin{theorem}\label{3-eq}
Let $H$ be a Hilbert space, $K=3$ or $K=4$,   
and let $L_1,\dots,L_K$ be closed subspaces of $H$; denote $L=\bigcap_{k\in [K]} L_k$.
The following statements are equivalent:
\begin{enumerate}
\item[(i)] $L_1^{\perp}+\dots+L_K^{\perp}$ is {\em not} closed in $H$;
\item[(ii)]   for every sequence $\{a_n\}\in c_0$, there exists 
a starting point $z\in H$ so that    $|T^nz-P(\bigcap L_k)z|\geq a_n$; 
\item[(iii)]   there exist    closed subspaces  $\tilde L_k\subset L_k\cap L^{\perp}$, a starting point $z_0\in H$  and a sequence of indices $k_1, k_2,\dots \in  [K]$ containing each $k\in [K]$ infinitely often, so that the sequence of iterates defined by $z_n= P(\tilde L_{k_n}) z_{n-1}$ does not converge in norm.
\end{enumerate}
\end{theorem}
\begin{proof}
The equivalence of (i) and (ii) is due to \cite{BDH, DH, BaGM1, BaGM2}.
That (i) implies (iii) follows from Lemma~\ref{BB} and Corollary~\ref{inspaces} applied to $L_k\cap L^{\perp}$. 

Assume (iii) is satisfied. 
Corollary~\ref{n34} implies that $\tilde L_1^{\perp}+\dots+\tilde L_K^{\perp}$ is not closed. By Lemma~\ref{BB} (iii),
there exists a normalized weakly-null sequence $\{w_n\}$    so that 
$$
\lim_{n\to\infty}\max\{\dist(w_n,L_k):\, k\in  [K]\}\leq \lim_{n\to\infty}\max\{\dist(w_n,\tilde L_k):\, k\in  [K]\}=0.
$$
Since $\tilde L_k\subset L^{\perp}$ for all $k\in [K]$, we can assume that  $\{w_n\}$ is in $L^{\perp}$ as well.
By Lemma~\ref{BB},  $L_1^{\perp}+\dots+L_K^{\perp}$ is  not closed.
\end{proof}

Here is a positive version of the above equivalence.

\begin{theorem}\label{3-eq-positive}
Let $H$ be a Hilbert space,  $K=3$ or $K=4$,   
and let $L_1,\dots,L_K$ 
  be closed subspaces of $H$, $\bigcap_{k\in [K]} L_k=\{0\}$.
The following statements are equivalent:
\begin{enumerate}
\item[(i)] $L_1^{\perp}+\dots+L_K^{\perp}$   is   closed in $H$;
\item[(ii)]  if $T=P_{L_K} \dots P_{L_1}$, then $\|T^n\|\leq Cr^n$ for some $r\in [0,1)$;
\item[(iii)]   for all closed subspaces  $\tilde L_i\subset L_i\subset H$,  each starting point $z_0\in H$  and every sequence of indices $k_1, k_2,\dots \in  [K]$, the sequence of iterates defined by $z_n= P({\tilde L_{k_n}}) z_{n-1}$ does  converge in norm.
\end{enumerate}
\end{theorem}
\begin{proof}
The equivalence of (i) and (ii) is due to \cite{BDH, DH, BaGM1, BaGM2}; explicit values for $C$ and $r$ in (ii) can be found in
     \cite{BaGM2} and \cite{PRZ}. The rest is contained in Theorem~\ref{3-eq}.
\end{proof}

\section{$K\geq 4$ spaces $L_1,\dots, L_K$}\label{K5}

In this section we give a version of Theorem~\ref{3-eq}  which is suitable for $K\geq 4$ spaces. We give a  condition   involving Johnson graphs which is equivalent to the existence of 
  a divergent product of projections.

The assumption of Corollary~\ref{inspaces} implies  by Lemma~\ref{seqeq} that  $\bigcap_{i\in A} L_{i}$ is infinite dimensional, or that $\sum_{i\in A} L_{i}^{\perp}$ is not closed for each $A\in \binom{[K]}{4}$.
  This in turn implies that (ii) of the following theorem is satisfied for $V=\binom{[K]}{4}$. Hence 
the implication (ii)$\Rightarrow$(i) of the following theorem generalizes Corollary~\ref{inspaces}. 

\begin{theorem}\label{hypothesis}
Let $H$ be a Hilbert space, $4\leq K$,   
and let $L_1,L_2,\dots, L_K$ be closed subspaces of $H$.
The following statements are equivalent:
\begin{enumerate}
\item[(i)] there exist a starting point $z_0\in H$, a sequence of indices $k_1, k_2,\dots \in [K]$ containing each of the numbers $1,\dots, K$ infinitely often, and    for each $k\in [K]$, there exists   a closed subspace  $\tilde L_k\subset L_k\subset H$,   so that the sequence of iterates defined by $z_n= P(\tilde L_{k_n}) z_{n-1}$ does not converge in norm;
\item[(ii)]    there exists $V\subset \binom{[K]}{4}$ so that $\bigcup V=[K]$, for $E=\{\{A,B\}\in \binom V2:\, |A\cap B|=3\}$ the graph $G=(V,E)$ is connected, and if $A\in V$, then $\bigcap_{i\in A} L_{i}$ is infinite dimensional, or $\sum_{i\in A} L_{i}^{\perp}$ is not closed.
\end{enumerate}
\end{theorem}
\begin{proof}
The implication (i)$\Rightarrow$(ii)  is the statement of Theorem~\ref{johnson}.

The proof of (ii)$\Rightarrow$(i) is a refined version of the proof of Corollary~\ref{inspaces}. Inductively, we will build a structure as in the assumptions of Proposition~\ref{squeorthoblocs}.

For   $\eps_i=200^{-i}$  we choose $N_i\in \NN$, $\psi^i\in \SSS_3$   according to  Proposition~\ref{squeorthoblocs}  and define
$\beta_n=(\eps_{n+1}/|\psi^{n+1}|)^2/8$.

We choose a decreasing sequence $\{\alpha_n\}$ of numbers satisfying $0<\alpha_n< 1/(K\cdot200^n)$ and $2\sum_n^{\infty} \alpha_i<\beta_n$, which are also small enough to accommodate the following property:
Assume  $\{w_n\}$ is an orthonormal sequence, $x_n\in H$ satisfy $|w_n-x_n|\leq \alpha_n$ for all $n\in \NN$ and a linear mapping $R$ is defined by $R(w_n)=x_n$. Then  
\begin{equation}\label{8g}
\|P(R(E))-P(E)\|\leq  \beta_n
\end{equation}
 for all subspaces $E\subset H_n:=\bigvee_{j=n}^{n+2N_n}w_j$ and for all $n\in \NN$.

We choose an infinite walk in $G$, namely a sequence  of vertices $A_i\in V$ so that $\{A_i,A_{i+1}\}\in E$ and $\bigcup_{i=n}^{\infty} A_i=[K]$ for all  $n\in \NN$.

Next we repeatedly use  (i)$\Rightarrow$(iii) of  Lemma~\ref{seqeq} to inductively construct an orthonormal sequence $\{w_n\}$. We block the sequence $\{w_n\}$ into blocks of length $2N_i$  so that all vectors in the $i$-th block  are very near to all  the spaces  $L_k$  with  $k\in A_i$.
 
To start the construction we choose an orthonormal  family
$\{w_n\}$ so that 
\begin{equation}\notag 
\max\{\dist(w_n,L_k):\, k\in A_1\}<\alpha_n
\end{equation}
for all  $n\in [2N_1]$.  For each $n\in [2N_1]$ and $k\in A_1$, we choose 
$w^k_n\in L_k$ so that
$
|w_n-w^k_n|\leq  \alpha_n
$.
In the second step, we enlarge the  orthonormal  family
$\{w_n\}$ so that  
\begin{equation}\notag 
\max\{\dist(w_n,L_k):\, k\in A_2\}<\alpha_n
\end{equation}
for all  $n\in \{2N_1+1,\dots, 2N_1+2N_2\}$. 
For each such $n$ and $k \in A_2$, we  choose $w_n^{k} \in L_k$ so that  
$
|w_n-w^k_n|\leq  \alpha_n
$.  

We define $E_1=\vee\{w_n\}_{n=1}^{2N_1+1}$ and $e_1=w_1$, $e_2=w_{2N_1+1}$. 

As with the second step, in the third one we enlarge the orthogonal family $\{w_n\}$
  to be of size  $2(N_1+N_2+N_3)$, and afterwards set 
$$
E_2=\vee\{w_n:\, n\in\{2N_1+1, \dots,2(N_1+N_2)+1\}\}
$$ 
and  $e_3=w_{2(N_1+N_2)+1}$.  In the next step we define
$$
E_3=\vee\{w_n:\, n\in\{2(N_1+N_2)+1, \dots,2(N_1+N_2+N_3)+1\}\},
$$ 
$e_4=w_{2(N_1+N_2+N_3)+1}$,
and so on. Moreover, during the above process, when choosing every single $w_n$ we make sure that 
\begin{equation}\label{13g}
\langle w_n,g\rangle=0 \mbox{  for all } g\in G_n=\bigvee \{w_m, w_m^k:\, m<n,\, k\in [K]\}
\end{equation}
with $|g|=1$.
 In the above set $G_n$ we include only those $w_m^k$  which were defined, that is, where $k$ is contained in the appropriate set $A_i$.  

We choose  $X_i,Y_i\subset E_i$  as in Proposition~\ref{squeorthoblocs}. 

For each 
  $k\in [K]$, we define a linear mapping 
$S^k:\spn\{w_n\}\to H$ as follows.
First we set  $S^k(w_n)=w_n$ for all $n\in \NN$. 
If    $w_n^k$ has already been  chosen for some $n$, then  we   
re-define $S^k(w_n)=w_n^k$.   
According to (\ref{8g}), $S^k$   in place of  $S$, $T$, or $U$ satisfies (viii) of Proposition~\ref{squeorthoblocs}.

Let $\{A_i\}$ be the infinite walk in $G$ chosen above. For
  $k\in [K]$, let $\{s_i^k\}_{i=0}^{\infty}$ be the   sequence  of $X,Y,Z,0$ as described  at the end of Section~\ref{SJ}.  If $k\in A_i\cap A_{i+1}$, that is, when $s_i^k\neq 0$, then $S^k(E_i)\subset L_k$.
 
Next we define a sequence of subspaces 
$L_{k,i}\subset L_k$ as described below and set
$$
\tilde L_k=\bigvee_{i\in \NN}  L_{k,i}\subset L_k.
$$

To start with, we set 
$L_{k,i}=\{0\}$ for all $i\in \NN$. Some of the subspaces we re-define as explained below.
$$
L_{k,i}=
\begin{cases}
S^k(X_i),  \mbox{ if } s_i^k=X \\
S^k(Y_i),  \mbox{ if } s_i^k=Y  \mbox{ and $i$ is odd}  \\
S^k(Y_i),  \mbox{ if } s_i^k=Z   \mbox{ and $i$ is even.}  
\end{cases}
$$
If $i$ is even and 
$s_{i-1}^k=s_i^k=s_{i+1}^k=Y$,
then  $L_{k,i}$ remains $\{0\}$. If   the neighbor  $s_{i-1}^k$ or $s_{i+1}^k$ of $s_{i}^k$ is $0$, we redefine $L_{k,i}$ as follows. If $i$ is odd and $s_i^k=Z$, then we proceed similarly:
$$  
L_{k,i}=
\begin{cases}
S^k(\vee\{e_j:\, j\in \{i,i+1\} \mbox{ and } (s^k_{j-1}=0 \mbox{ or } s^k_{j}=0)\}  ),  \mbox{ if }  s_i^k=Y,   \mbox{ and $i$ is even,}  \\
S^k(\vee\{e_j:\, j\in \{i,i+1\} \mbox{ and } (s^k_{j-1}=0 \mbox{ or } s^k_{j}=0)\}),  \mbox{ if }   s_i^k=Z,   \mbox{ and $i$ is  odd.}  
\end{cases}
$$
Let $j\in \NN$ be given.
Recall from the very end of Section~\ref{SJ}  that 
  there is exactly one $k_{j,X}\in [K]$ so  that $s_j^{k_{j,X}}=X$. Similarly, there is exactly one $k_{j,Y}\in [K]$ and exactly one $k_{j,Z}\in [K]$ so  that $s_j^{k_{j,Y}}=Y$ and $s_j^{k_{j,Z}}=Z$.
  We apply Proposition~\ref{squeorthoblocs}  with $S$, $T$, and $U$ replaced by $S^k$ with $k\in 
A_j\cap A_{j+1}=\{k_{j,X}, k_{j,Y},k_{j,Z}\}$ and $\tilde X$, $\tilde Y$, $\tilde Z$ replaced by $\tilde L_{k_{j,X}}, \tilde L_{k_{j,Y}}, \tilde L_{k_{j,Z}}$, respectively.
Then (ix) of Proposition~\ref{squeorthoblocs} is clearly satisfied for 
$\tilde X_i=L_{k_{j,X},i}$, $\tilde Y_i=L_{k_{j,Y},i}$, and 
$\tilde Z_i=L_{k_{j,Z},i}$.

From the inductive   construction of $E_i$'s and from  (\ref{8g}), (\ref{13g}), Lemmma~\ref{justforme3}  and Lemma~\ref{justforme1} it follows that (xiii) of Proposition~\ref{squeorthoblocs} is satisfied  as well.
From the definition of the spaces $L_{k,i}$  it follows 
that $\tilde X_j=L_{k_{j,X},j}$ satisfies (x) of Proposition~\ref{squeorthoblocs}.
If $j$ is odd,  then
$\tilde Y_j=L_{k_{j,Y},j}$ satisfies (xi) of Proposition~\ref{squeorthoblocs} and $\tilde Z_j=L_{k_{j,Z},j}$ satisfies (xii) of Proposition~\ref{squeorthoblocs}.
Hence by Proposition~\ref{squeorthoblocs}, we have
\begin{equation}\notag
|\psi^j(\tilde L_{k_{j,Z}}, \tilde L_{k_{j,X}},\tilde L_{k_{j,Y}})e_j-e_{j+1}|<30\eps_j.
\end{equation}
for odd $j$.
If $j$ is even, then  we switch $Y$ and $Z$.

We define $z_0=e_1$ and 
write for short $R_i=\psi^i(\tilde L_{k_{i,Z}},\tilde L_{k_{i,X}},\tilde L_{k_{i,Y}})$; again, we switch $Y$ and $Z$ if $j$ is even.
 Since $\|R_i\|\leq 1$, we have by induction
 \begin{equation}\label{glue}
 \begin{split}
 \|R_n&R_{n-1}\dots R_1 e_1 -e_{n+1}\| \\
 &\leq \|R_nR_{n-1}\dots R_2(R_1 e_1 -e_2)\|+ \|R_nR_{n-1}\dots R_2e_2-e_{n+1}\| \\
 &\leq 30\eps_1+\|R_nR_{n-1}\dots R_3(R_2 e_2 -e_3)\|+ \|R_nR_{n-1}\dots R_3e_3-e_{n+1}\| \\
& \leq 30\eps_1+30\eps_2+\dots+\|R_n e_n-e_{n+1}\| \leq 30(\eps_1+\dots+\eps_n)\\
& \leq 60(\eps_1+\dots+\eps_n)\leq 1/2
 \end{split}
 \end{equation}
 for all $n\in \NN$. Since $\{e_i\}$ is an orthonormal sequence, the norm-limit
 $\lim_{n\to \infty} R_nR_{n-1}\dots R_1 e_1$ does not exist.
 
 To make sure that all the $K$ spaces $\tilde L_1,\dots,\tilde L_K$ appear in the product infinitely often, we re-define
 $R_i=\psi_i (\tilde L_{k_{i,Z}},\tilde L_{k_{i,X}},\tilde L_{k_{i,Y}})\Pi_{k\in A_i}P(\tilde L_k)$. Then
\begin{equation}\notag
\begin{split}
|R_i e_i-e_{i+1}|\leq &|\psi^i(\tilde L_{k_{i,Z}},\tilde L_{k_{i,X}},\tilde L_{k_{i,Y}})e_i-e_{i+1}| \\&+\|\psi^i(\tilde L_{k_{i,Z}},\tilde L_{k_{i,X}},\tilde L_{k_{i,Y}})\|\cdot |\Pi_{k\in A_i}P(\tilde L_k)e_i-e_i| \leq 60\eps_i.
\end{split}
\end{equation}
Now the statement is again seen to follow from
(\ref{glue}). 
\end{proof}

In Example~\ref{not3} we exhibit  that reducing the 4-tuple condition to a 3-tuple does not work.

\section{Examples}\label{examples}

We dedicate this section to examples which show where the previously presented results cannot be strengthened or modified.

Let $L_1,\dots, L_K$  be closed subspaces of a Hilbert space $H$. Suppose there is a starting point $z_0\in H$  and a sequence of indices $k_1, k_2,\dots \in \{1,\dots ,K\}$ containing each of the numbers $1,\dots, K$ infinitely often  so that the sequence of iterates defined by $z_n= P(L_{k_n}) z_{n-1}$ does not converge in norm.

In Theorem~\ref{johnson} we have shown that there is a connected induced
subgraph $G=(V,E)$ of $J(K,4)$ so that $\bigcup V=[K]$ and if $A\in V$, then $\bigcap_{i\in A} L_i$ is infinite dimensional  or $\sum_{i\in A} L_{i}^{\perp}$ is not closed.
Here  we construct for   each subgraph $G$ as above   a corresponding  configurations of subspaces of $H$.

\begin{proposition}~\label{johnbio}
Let $H$ be an infinite dimensional Hilbert space. Let $K\geq 4$ and  $V\subset  \binom{[K]}4$ be so that $\bigcup V=[K]$, and let $\alpha:V\to \{0,1\}$.  Then there  exist closed subspaces $L_1,\dots,L_K\subset H$ with the following properties.  If $A\in V$ and $\alpha(A)=0$, then $\bigcap_{i\in A}L_{i}$ is infinite dimensional. If $A\in V$ and $\alpha(A)=1$, then $\sum_{i\in A} L_{i}^{\perp}$ is not closed.
If $B\in \binom{[K]}4\setminus V$, then $\bigcap_{i\in B}L_{i}$ is finite dimensional  and   $\sum_{i\in B} L_{i}^{\perp}$ is   closed.
\end{proposition}
\begin{proof}
Choose $|V|$ pairwise orthogonal closed infinite dimensional subspaces $H_A$ of $H$ so that 
$$
H=\bigoplus_{A\in V}H_A.
$$
In each $H_A$ choose 5 pairwise orthogonal orthonormal sequences
$\{e^A_n\}_{n\in \NN}$ and $\{e^A_{i,n}\}_{n\in \NN}$, $i\in A$.  For $k\in [K]$, we define
$$
L_k=\bigvee_{n\in \NN}\{e^A_n+\alpha(A) e^A_{k,n}/n:\, k\in A\in V\}.
$$
Note   that if $k\notin A$, then
$L_k\perp H_A$.

Assume $A\in V$. If $\alpha(A)=0$, then $\vee\{e_n^A\}_{n\in \NN}\subset 
\bigcap_{i\in A}L_{i}$.
If $\alpha(A)=1$, then $\bigcap_{i\in A}L_{i}=\{0\}$ and $\max\{\dist(e_n^A,L_i:\, i\in A\}\leq 1/n$.
Hence  the subspace $\sum_{i\in A} L_{i}^{\perp}$ is not closed by Lemma~\ref{BB}. 

Let $B\in  \binom{[K]}4\setminus V$ be given.  Assume  $0\neq w\in\bigcap_{i\in B}L_{i}$. Then there is $A\in V$ so that $P(A)w\neq 0$.
Choose $j\in B\setminus A$. 
Then $L_j \perp H_A$, hence $w\notin L_j$, which is a contradiction.
Thus $\bigcap_{i\in B}L_{i}=\{0\}$.

Assume $\sum_{i\in B} L_{i}^{\perp}$ is not closed for some $B\in  \binom{[K]}4\setminus V$. By Lemma~\ref{BB}, there is an orthornormal sequence $\{w_n\}$ so that
$$
\lim_{n\to\infty}\max\{\dist(w_n,L_i):\, i\in  B\}=0.
$$ 
  Write each $w_n=\sum_{A\in V}P_A(w_n)$, where $P_A$ is the orthoprojection onto $H_A$. For each $n\in \NN$, there is an $A\in V$ for which $|P_A(w_n)|\geq 1/|V|$. Since $V$ is finite, by replacing $\{w_n\}$ with a suitable subsequence, we can assume this $A$ is the same for all $n\in \NN$. Choose $j\in B\setminus A$. 
 Then $L_j \perp H_A$, hence $\dist(w_n,L_j)\geq |P_A(w_n)|\geq 1/|V|$ 
 for all $n\in \NN$,  which is a contradiction.
 \end{proof}

Bauschke showed in \cite{B}    that 
if $\sum_{j\in J}L_j^{\perp}$ is closed for each nonempty 
$J\subset [K]$, then $\{z_n\}$ converges in norm. Reich and Zaslavski show in \cite{RZ}, that $L_1^{\perp}+\dots+L_{K}^{\perp}$ being closed is highly {\em not} hereditary with respect to subsets $J\subset [K]$.
A simple example for $K=3$  spaces follows.

\begin{example}\label{Bk}
Let $H$ be an infinite dimensional Hilbert space. Then there exist three closed subspaces $L_1,L_2,L_3$ of $H$ so that $L_1^{\perp}+L_2^{\perp}+L_3^{\perp}$ is closed, but $L_1^{\perp}+L_2^{\perp} $ is {\em not} closed in $H$.
\end{example}
\begin{proof}
It is enough to construct the example in a separable Hilbert space,   so assume $H=\ell_2$ and write 
$$
H=\R^3\oplus\R^3\oplus\R^3\oplus\dots,
$$
where   the $j$-th copy of $\R^3$ has an orthonormal basis  $\{e_{3j}, e_{3j+1}, e_{3j+2} \}$.
Define
\begin{equation}\notag
\begin{split}
L_1&=\bigvee\{e_{3j+1}:\, j\in \NN\} \\
L_2&=\bigvee\{e_{3j+1}+e_{3j+2}/j:\, j\in \NN\} \\
L_3&=\bigvee\{e_{3j} :\, j\in \NN\}.
\end{split}
\end{equation}
Clearly, $L_1\cap L_2=\{0\}$, and $L_1^{\perp}+L_2^{\perp}+L_3^{\perp}=H$. 
  Since 
$$
\lim_{n\to \infty}\max\{\dist(e_{3n+1},L_k):\, k\in \{1,2\}\}=0,
$$
it follows from  Lemma~\ref{BB} that  the subspace $L_1^{\perp}+L_2^{\perp}$ is not closed.
\end{proof} 

On the one hand, the condition  that $L_1^{\perp}+L_2^{\perp}+L_3^{\perp}$   is closed  is,   according to Theorem~\ref{3-eq-positive}, enough to ensure norm convergence even when projecting onto   subspaces of the given spaces, but on the other hand, it is   not really necessary
to ensure 
norm convergence when projecting just on the $L_k$'s themselves, as  the next  example  shows.  

\begin{example}\label{slownono}
Let $H$ be an infinite dimensional Hilbert space. Then there exist three closed subspaces $L_1,L_2,L_3$ of $H$ so that 
\begin{enumerate}
\item[(i)] $L_1^{\perp}+L_2^{\perp}+L_3^{\perp}$ is {\em not} closed in $H$;
\item[(ii)] for every $0\neq z_0\in H$ and $k_1, k_2,\dots \in \{1,2,3\}$,   the sequence of iterates defined by $z_n= P(L_{k_n}) z_{n-1}$   converges in norm (possibly very slowly).  
\end{enumerate}
\end{example}
\begin{proof}
It is enough to construct the example in a separable Hilbert space, so assume $H=\ell_2$ and write 
$$
H=\R^2\oplus\R^2\oplus\R^2\oplus\dots,
$$
where the $j$-th copy of $\R^2$ has an orthonormal basis $\{e_{2j-1}, e_{2j}\}$.
Define
\begin{equation}\notag
\begin{split}
L_1&=\bigvee\{e_{2j-1}:\, j\in \NN\} \\
L_2&=\bigvee\{e_{2j-1}+e_{2j}/j:\, j\in \NN\} \\
L_3&=\bigvee\{e_{2j-1}+2e_{2j}/j:\, j\in \NN\}.
\end{split}
\end{equation}
It is easy to see that $L_1\cap L_2\cap L_3=\{0\}$. Since 
$$
\lim_{n\to \infty}\max\{\dist(e_{2n-1},L_k):\, k\in \{1,2,3\}\}=0,
$$
it follows from  Lemma~\ref{BB} that the subspace $L_1^{\perp}+L_2^{\perp}+L_3^{\perp}$ is not closed.

Let $z_0\in H$ and $k_n\in \{1,2,3\}$ be given.
Write $z_0=\sum_{j=1}^{\infty}z_0^j$, where $z_0^j$ is the projection of $z_0$ onto the $j$-th copy of $\R^2$.
For every $j\in \NN$, define the sequence $z^j_n= P(L_{k_n}) z^j_{n-1}\in \bigvee \{e_{2j-1}, e_{2j}\}$. This is a sequence of iterates  of projections on the three   lines
$\vee e_{2j-1}$, $\vee (e_{2j-1}+e_{2j}/j)$ and $\vee (e_{2j-1}+2e_{2j})/j$ in the $j$-th copy of $\R^2$,  which intersect at the origin. Hence 
$$
\lim_{n\to \infty} z^j_n=0.
$$
By,  {\em e.g.},  \cite{KKM}, 
$
z_n=\sum_{j=1}^{\infty}z_n^j
$. Hence for every $\eps>0$, there exists $N\in \NN$  so that for all $n$ large enough,
$$
|z_n|^2=\sum_{j=1}^{\infty} |z_n^j|^2\leq\sum_{j=1}^N |z_n^j|^2+\sum_{j=N+1}^{\infty}|z_0^j|^2<\eps.
$$
\end{proof}

Deutsch and Hundal ask in \cite{DH} if just $L_1^{\perp}+\dots+L_{K}^{\perp}$ being closed 
is enough to guarantee the norm convergence of $\{z_n\}$. In Theorem~\ref{3-eq-positive} we  show this is indeed the case for $K=3$ and $K=4$.
For $K\geq 5$ the next example shows
that this is not the case in general. In other words, 
Proposition~\ref{34} is no longer true for  $|I|\geq 5$.

\begin{example}\label{5}
Let $H$ be an infinite dimensional Hilbert space. 
For every $K\geq 5$, there exist closed subspaces
$L_1,\dots, L_K\subset H$,   so that $\bigcap_{k\in [K]} L_k=\{0\}$
with the following properties:
\begin{enumerate}
\item[(i)] there is a starting point $z_0\in H$  and a sequence of indices $k_1, k_2,\dots \in [K]$ containing each of the numbers $1,\dots, K$ infinitely often  so that the sequence of iterates defined by $z_n= P(L_{k_n}) z_{n-1}$ does not converge in norm;
\item[(ii)] the subspaces $L_4,L_5,\dots, L_K$ are pairwise orthogonal;
\end{enumerate}
in particular, $L_1^{\perp}+\dots+L_K^{\perp}=H$.  
\end{example}
\begin{proof}
For $\eps_i=4^{-i}$ consider the spaces $X,Y, Z$ as defined in Lemma~\ref{orthoblocs}, and define $z_0=e_1$, $L_1=X$, $L_2=Y$, $L_3=Z$. For $k\in \{4,\dots, K\}$, define
\begin{equation}\notag
\begin{split}
L_k&=\bigvee\{e_i:\, i=k \mod (K-3)\} \\
A_i&=\Psi_i(Z,X,Y)P(L_k) \mbox{ if }  i=k \mod (K-3).
\end{split}
\end{equation}
Then $L_4,L_5,\dots, L_K$ are pairwise orthogonal; hence
$$
L_1^{\perp}+\dots+L_K^{\perp}\supset L_4^{\perp}+L_5^{\perp}=H.
$$
Also
$P(L_k)e_i=e_i$ if $i=k \mod (K-3)$, and 
$$
|A_ie_i-e_{i+1}|=|\Psi_i(Z,X,Y)P(L_k)e_i-e_{i+1}|=|\Psi_i(Z,X,Y)e_i-e_{i+1}|\leq 4^{-i}.
$$
Since $\|A_i\|\leq 1$, we have by induction
 \begin{equation}\notag
 \begin{split}
 \|A_n&A_{n-1}\dots A_1 e_1 -e_{n+1}\| \\
 &\leq \|A_nA_{n-1}\dots A_2(A_1 e_1 -e_2)\|+ \|A_nA_{n-1}\dots A_2e_2-e_{n+1}\| \\
 &\leq \eps_1+\|A_nA_{n-1}\dots A_3(A_2 e_2 -e_3)\|+ \|A_nA_{n-1}\dots A_3e_3-e_{n+1}\| \\
& \leq \eps_1+\eps_2+\dots+\|A_n e_n-e_{n+1}\|\leq   1/2
 \end{split}
 \end{equation}
 for all $n\in \NN$. Since $\{e_i\}$ is an orthonormal sequence, the norm-limit
 $\lim_{n\to \infty} A_nA_{n-1}\dots A_1 e_1$ does not exist.
\end{proof}

To conclude, we give in the above spirit an example showing  that for $K\geq 4$,
the Johnson graph $J(K,3)$ can not replace $J(K,4)$ in Theorem~\ref{hypothesis}.

\begin{example}\label{not3}
Let $H$ be an infinite dimensional Hilbert space and let $4\leq K$. Then there exist  
$L_1,L_2,\dots, L_K$,   closed subspaces of $H$, $\bigcap_{k\in [K]} L_k=\{0\}$, so that 
\begin{enumerate}
\item[(i)]  for all   closed subspaces  $\tilde L_i\subset L_i\subset H$, any starting point $z_0\in H$  and any sequence of indices $k_1, k_2,\dots \in [K]$ containing each of the numbers $1,\dots, K$ infinitely  often, the sequence of iterates defined by $z_n= P(\tilde L_{k_n}) z_{n-1}$ does  converge in norm;
\item[(ii)]    there exists $V\subset \binom{[K]}{3}$, so that $\bigcup V=[K]$, for $E=\{\{A,B\}\in \binom V2:\,  A\neq B, \, |A\cap B|=2\}$ the graph $G=(V,E)$ is connected, and if $A\in V$, then $\sum_{i\in A} L_{i}^{\perp}$ is not closed.
\end{enumerate}
\end{example}
\begin{proof}
It is enough to construct the example in a separable Hilbert space, so assume $H=\ell_2$ and write again
$$
H=\R^2\oplus\R^2\oplus\R^2\oplus\dots,
$$
where the $j$-th copy of $\R^2$ has an orthonormal basis $\{e_{2j-1}, e_{2j}\}$.
For $i\in \{3,\dots,K\}$, define
\begin{equation}\notag
\begin{split}
L_1&=\bigvee\{e_{2j-1}:\, j\in \NN\}, \\
L_2&=\bigvee\{e_{2j-1}+e_{2j}/j:\, j\in \NN\}, \\
L_i&=\bigvee\{e_{2j-1}+2e_{2j}/j:\, j\in \NN,\  j=i \mod (K-2)\}.
\end{split}
\end{equation}
Then $L_1\cap L_2=\{0\}$.
The spaces $L_3,\dots, L_K$ are pairwise orthogonal, hence 
$\bigcap_{k\in A}L_k=\{0\}$, and $\sum_{i\in A} L_{i}^{\perp}=H$ for each $A\in \binom{[K]}{4}$.
By Proposition~\ref{34} and Lemma~\ref{BB},  
for all closed subspaces  $\tilde L_i\subset L_i\subset H$,  each starting point $z_0\in H$  and every sequence of indices $k_1, k_2,\dots \in  [K]$ containing  all  of the numbers in $[K]$ infinitely often, the sequence of iterates defined by $z_n= P({\tilde L_{k_n}}) z_{n-1}$ does  converge in norm.

Define $V=\{\{1,2,i\}:\, i\in \{3,\dots, K\}\}$. Then $\bigcup V=[K]$ and 
$|A\cap B|=2$ for all different $A,B\in V$.
If $A\in V$, then $\sum_{i\in A} L_{i}^{\perp}$ is not closed. 
\end{proof}

\section{Appendix}\label{appen}

For an easy reference, this section  contains a   few   elementary tools we use in the above text.

\begin{lemma}\label{justforme3}
Let $\{f_n\}$ be an orthonormal sequence in a Hilbert space $H$. Let $\{\alpha_n\}$ be a sequence of positive numbers so that $\sum_{n=1}^{\infty}\alpha_n<\beta<1/4$. Let $w_{n}\in H$ be so that         $|w_{n}-f_{n}|\leq \alpha_n$ for all $n\in \NN$. Let $T:\vee\{f_n\}\to \vee\{w_n\}$ be a linear operator defined by $Tf_n=w_n$. Then
\begin{enumerate}
\item[(i)] $T$ is a surjective isomorphism;
\item[(ii)] $\|T\|\leq 1+\beta$;
\item[(iii)] $\|T^{-1}\|\leq 1/(1-\beta)$;
\item[(iv)] $|Tx-x|\leq \beta |x|$ for $x\in \vee\{f_n\}$.
\end{enumerate}
\end{lemma}
\begin{proof}
Let $x=\sum_{n=1}^Na_nf_n$, $|x|\leq 1$, be given. Then
$$
|Tx-x|\leq \sum_{n=1}^N|a_n|\cdot|f_n-w_n|\leq |x|\sum_{n=1}^N\alpha_n\leq\beta |x|.
$$
Hence $(1-\beta)|x|\leq |Tx|\leq (1+\beta)|x|$ and $T:\spn\{f_n\}\to  {\spn}\{w_n\}$ satisfies (ii)-(iv).
Thus $T$ is an isomorphism of $\vee\{f_n\}$ onto $\vee\{w_n\}$.
\end{proof}

Next we observe that within a normalized weakly-null sequence  there exist
almost pairwise orthogonal blocks of almost orthogonal vectors.

\begin{lemma}\label{justforme0}
Let $\{w_n\}$ be a  normalized weakly-null sequence in a Hilbert space $H$  and let $0<a_n\leq 1/2$ be so that $\lim a_n=0$. Let $V\subset H$ be a finite dimensional subspace. 
Then there exists an orthonormal sequence $\{e_i\}_{i=1}^{\infty}$ in $V^{\perp}$ and a subsequence
$\{w_{n_i}\}$ of $\{w_n\}$ so that   for all $i\in \NN$,
\begin{enumerate}
\item[(i)]  $|w_{n_i}-e_i|\leq a_i$;
\item[(ii)] $|\langle f,g\rangle| \leq a_i$ for all $f\in F_i=V\vee\{e_j, w_{n_j}:\, 1\leq j\leq i\}$ and  $g\in G_i= \vee\{w_{n_j}:\, i<j\}\cup \vee\{e_j:\, i<j\}$  with $|f|=|g|=1$.
\end{enumerate}
\end{lemma}
\begin{proof}
We can assume that $a_{n+1}\leq a_n/2$   and choose a decreasing sequence  $\alpha_n>0$ so that   $\sum_{i=n}^{\infty} \alpha_i<a_n$ for all $n\in \NN$.
We define $e_0:=w_{n_0}:=w_1$ and continue by induction.  
In the $(i+1)$-st step we choose an orthonormal basis 
$\{f_1,\dots, f_m\}\supset \{e_0,\dots, e_i\}$   of  $Z=V\vee\{e_j, w_{n_j}:\, 0\leq j\leq i\}$.   Next, we choose
  $M\geq m$ large enough for our considerations below and $n_{i+1}>n_i$ so that
$$
w_{n_{i+1}}\in  \bigcap_{j=1}^m\{x\in H:\, |\langle x,f_j\rangle|<\alpha_{i+1}/M\}.
$$
The projection $e$ of $w_{n_{i+1}}$ onto $Z^{\perp}$ has norm at least
$\sqrt{1-\alpha_{i+1}^2/M}$. Define $e_{i+1}=e/|e|$. Then $|e_{i+1}-w_{n_{i+1}}|<\alpha_{i+1}<a_{i+1}$  if $M>0$ is large enough. 

 To verify (ii), let $i\in \NN$ and $f\in F_i$ with $|f|=1$ be given. 
 If $g\in \vee\{e_j:\, i<j\}$, then $\langle f,g\rangle=0$.   Let $T:\vee\{e_j:\, i<j\}\to \vee\{w_{n_j}:\, i<j\}$ be a linear operator defined by 
 $Te_j=w_{n_j}$. Let 
$g\in \vee\{w_{n_j}:\, i<j\}$ with $|g|\leq 1$ be given. Then by Lemma~\ref{justforme3}, 
$$
|\langle f,g\rangle|\leq |\langle f,T^{-1}g\rangle|+|g-T^{-1}g|= |T(T^{-1}g))-T^{-1}g|
\leq \frac{a_{i+1}}{1-a_{i+1}}\leq a_i.
$$

\end{proof}

If the angle between two closed subspaces of a Hilbert space is large, then the sum of the spaces is closed. Moreover, the projection on the sum is approximately equal the sum of the projections on the two spaces.

\begin{lemma}\label{justforme1}
Let $H$ be a Hilbert space,   $F$ and $G$
be closed subspaces of $H$ and  
let $0<\alpha<1/3$. Assume $|\langle f,g\rangle|\leq \alpha$ for all  $f\in F$ and  $g\in G$   with $|f|\leq 1$ and $|g|\leq 1$.
Then 
\begin{enumerate}
\item[(i)]  $\frac 12 B_{F\vee G}\subset B_F+B_G$;
\item[(ii)] $F\vee G= F+G$, hence $F+G$ is  closed;
\item[(iii)] $\|P(F)+P(G)-P(F\vee G)\|\leq 4\sqrt{\alpha}$;
\item[(iv)] $\|P(F)P(G)\|\leq \sqrt{\alpha}$.
\end{enumerate}
\end{lemma}
\begin{proof}
From the definition of the closed linear span of two linear spaces it follows that $F\vee G=\overline{F+G}$.
Notice also that $B_F+B_G$ is closed as a sum of two weakly compact sets. 
Let $x\in F\vee G$ be so that $|x|<1/2$.  By the above discussion, arbitrarily close to $x$  there exists  a point $y=f+g$, where    $f\in F$ and $g\in G$. Assume $|f|\leq |g|$. Then
$$
\frac 14>|y|^2\geq|f|^2+|g|^2-2\alpha |f|\cdot|g|\geq |g|^2(1-2\alpha);
$$
hence $|f|\leq |g|\leq 1$.

Statement (ii) follows immediately from (i). 

To show (iii), let $h\in H$ with $|h|\leq 1$ be given. First we show that $h-(P(F)h+P(G)h)$ is almost 
orthogonal to $F\vee G$, meaning that $P(F)h+P(G)h$ is a good candidate for an approximation of $P(F\vee G)h$.

Let $x\in F\vee G$, $|x|\leq 1$, be given.  According to (i), $x=f+g$ for some   $f\in F$ and $g\in G$  with
$|f|\leq 2$ and $|g|\leq 2$. Then
\begin{equation}\notag
\begin{split}
|\langle h-P(F)h-P(G)h,x\rangle|=&|\langle h-P(F)h-P(G)h,f+g\rangle| \\ 
\leq &\langle h-P(F)h,f\rangle+|\langle P(G)h,f\rangle| \\
&+\langle h-P(G)h,g\rangle+|\langle P(F)h,g\rangle|
  \\
&\leq 4\alpha.
\end{split}
\end{equation}
Hence
\begin{equation}\notag
\begin{split}
|P(F\vee G)h&-P(F)h-P(G)h|^2 
\\ &\leq|\langle P(F\vee G)h-h,P(F\vee G)h-P(F)h-P(G)h\rangle| \\
&\ \  +|\langle h-P(F)h-P(G)h,P(F\vee G)h-P(F)h-P(G)h\rangle  \\ &
\leq 12\alpha.
\end{split}
\end{equation}
To show (iv), let $h\in H$ with $|h|=1$ be given. 
Then
\begin{equation}\notag
\begin{split}
\|P(F)P(G)h\|^2=&\langle P(F)P(G)h,P(F)P(G)h\rangle \\
=&\langle P(F)P(G)h-P(G)h,P(F)P(G)h\rangle  \\
&+\langle P(G)h,P(F)P(G)h\rangle\leq \alpha.
\end{split}
\end{equation}
\end{proof}


\end{document}